\documentclass[11pt]{article}
\usepackage[utf8]{inputenc}
\usepackage[T1]{fontenc}
\usepackage{listings}

\usepackage{xcolor}
\definecolor{juliapurple}{RGB}{95,60,128}
\definecolor{juliagreen}{RGB}{56,152,38}
\definecolor{codegray}{RGB}{247,247,247}
\definecolor{darkteal}{RGB}{0,70,70}
\definecolor{juliaprompt}{RGB}{0,140,50}
\definecolor{Maroon}{RGB}{128,0,0}
\lstdefinelanguage{Julia}{
  morekeywords=[1]{
    using,import,function,end,return,if,else,elseif,for,while,in,do,
    struct,mutable,where,begin,let,const,nothing,true,false, NTuple, Float64, AbstractVector, Bool, push!, 
  },
  morekeywords=[2]{
    TriangulationCache, visualize, refine!, Sierpinski, Barycenter, Random, save
  },
  morekeywords=[3]{
    resolution, xlims, ylims, edges, buttons, strategy, is_complete, min_refinement_area
  },
  sensitive=true,
  morecomment=[l]{\#},
  morestring=[b]",
}
\usepackage[margin=1in]{geometry}
\usepackage{microtype}
\usepackage{amsmath,amsfonts,amssymb,amsthm,amsopn,mathtools}
\usepackage{thmtools}
\usepackage{accents}
\usepackage{bm}
\usepackage{graphicx}
\usepackage{xcolor}
\usepackage{nicematrix}
\usepackage{multirow}
\usepackage{makecell}
\usepackage{multicol}
\usepackage{blkarray}
\usepackage{subcaption}
\usepackage{paralist}
\usepackage{enumitem}
\setlist[enumerate]{leftmargin=.5in}
\setlist[itemize]{leftmargin=.5in}
\usepackage{algorithm}
\usepackage{algpseudocode}
\usepackage{verbatim}
\algrenewcommand\algorithmicrequire{\textbf{Inputs:}}
\algrenewcommand\algorithmicensure{\textbf{Output:}}

\usepackage{tikz}
\usepackage{tikz-cd}
\usetikzlibrary{
  decorations.pathreplacing,
  patterns,
  external,
  calc,
  shapes.multipart
}
\definecolor{fondo}{rgb}{0.898,0.996,0.898}
\pgfkeys{/tikz/.cd, K/.store in=\K, K=1}
\usepackage{todonotes}
\usepackage[normalem]{ulem}

\newcommand{\mycomment}[1]{}
\usepackage[numbers,sort]{natbib}
\usepackage{hyperref}
\usepackage[nameinlink,noabbrev]{cleveref}
\hypersetup{
  colorlinks=true,
  linkcolor=blue!55!black,
  citecolor=blue!55!black,
  urlcolor=blue!60!black,
  pdftitle={Alt's Problem},
  pdfauthor={Taylor Brysiewicz}
}

\theoremstyle{plain}
\newtheorem{theorem}{Theorem}[section]
\newtheorem{lemma}[theorem]{Lemma}
\newtheorem{proposition}[theorem]{Proposition}
\newtheorem{corollary}[theorem]{Corollary}

\theoremstyle{definition}

\newtheorem{example}[theorem]{Example}
\theoremstyle{remark}
\newtheorem{remark}[theorem]{Remark}

\crefname{hypothesis}{hypothesis}{hypotheses}
\Crefname{hypothesis}{Hypothesis}{Hypotheses}
\crefname{claim}{claim}{claims}
\Crefname{claim}{Claim}{Claims}
\crefname{fact}{fact}{facts}
\Crefname{fact}{Fact}{Facts}
\crefname{question}{question}{questions}
\Crefname{question}{Question}{Questions}
\crefname{subsection}{section}{sections}
\Crefname{subsection}{Section}{Sections}

\newcommand{\email}[1]{\href{mailto:#1}{\texttt{#1}}}
\title{Alt's Problem}
\author{Taylor Brysiewicz\thanks{University of Western Ontario, London, ON
  (\email{tbrysiew@uwo.ca}).}}
\date{}

\newcommand{\mydef}[1]{{\color{blue}#1}}

\makeatletter
\DeclareRobustCommand{\sqcdot}{\mathbin{\mathpalette\morphic@sqcdot\relax}}
\newcommand{\morphic@sqcdot}[2]{%
  \sbox\z@{$\m@th#1\centerdot$}%
  \ht\z@=.33333\ht\z@
  \vcenter{\box\z@}%
}
\makeatother
\makeatletter
\def\fooC#1{%
  \expandafter\newcommand\csname c#1\endcsname{\mathcal{#1}}}
\def\fooB#1{%
  \expandafter\newcommand\csname b#1\endcsname{\mathbb{#1}}}
\count@=0
\loop
  \advance\count@ 1
  \edef\y{\@Alph\count@}%
  \expandafter\fooC\y
  \expandafter\fooB\y
\ifnum\count@<26
\repeat
\makeatother
\begin{document}
\maketitle

\vspace{-0.3in} 

\begin{abstract}
We prove there are $1442$  four-bar coupler curves  through nine generic points in the plane, and thus resolve Alt's problem. We obtain this proof in three steps. First, we identify the space of coupler curves with a Zariski open subset of $\textrm{Gr}(3,6)$. Next, we formulate the polynomial system representing the nine-point path synthesis problem in these coordinates  and modify it to obtain the mixed volume $5538$. Finally, we prove that $4096$ of the branches of the generic sparse polynomial system with that support escape the torus in the sparse limit. Thus, we obtain an upper bound of $5538-4096=1442$ for the generic solution count. A lower bound of $1442$ is achieved via numerical certification on one instance. 
\end{abstract}

\section{Introduction}  
\label{sec:Introduction}
 The prototypical interpolation problem for a family $\mathcal F$ of varieties asks how many elements in $\mathcal F$ pass through $\dim(\mathcal F)$ generic points. For this work,  we consider \textit{coupler curves of four-bar linkages}, like the one shown below in \Cref{fig:coupler}. The corresponding interpolation problem is called the \textit{nine-point path synthesis problem}. Determining this count is \mydef{Alt's problem}, posed in 1923~\cite{Alt1923}. 

\vspace{-10pt}

\begin{figure}[htbp]
    \centering
\def\verticalstretch{0.4} 
\def\figurescale{0.45}    
\def\panelseparation{11.8}
\def\couplerAlong{1.0625}
\def\couplerNormal{0.9375}
\begin{tikzpicture}[
    scale=\figurescale, yscale=\verticalstretch,
    line cap=round,line join=round,
    pivot/.style={circle,fill=black,inner sep=1.2pt},
    every label/.style={font=\small}
]
    \pgfmathsetmacro{\thetaZero}{acos((sqrt(13)-1)/sqrt(10))}
    \foreach \branch in {-1,1} {
        \xdef\fourbarCurveCoordinates{}
        \foreach \sample in {0,...,320} {
            \pgfmathsetmacro{\thetaValue}{\thetaZero
                +(360-2*\thetaZero)*(1-cos((\sample/320)*180))/2}
            \pgfmathsetmacro{\cx}{sqrt(10)*cos(\thetaValue)}
            \pgfmathsetmacro{\cy}{sqrt(10)*sin(\thetaValue)}
            \pgfmathsetmacro{\distanceSquared}{26-8*\cx}
            \pgfmathsetmacro{\lambdaValue}{(\distanceSquared-18)/(2*\distanceSquared)}
            \ifnum\sample=0
                \def\muValue{0}
            \else\ifnum\sample=320
                \def\muValue{0}
            \else
                \pgfmathsetmacro{\muValue}{sqrt(max(0,
                    8/\distanceSquared-(\lambdaValue)^2))}
            \fi\fi
            \pgfmathsetmacro{\qx}{\lambdaValue*(4-\cx)+\branch*\muValue*\cy}
            \pgfmathsetmacro{\qy}{-\lambdaValue*\cy+\branch*\muValue*(4-\cx)}
            \pgfmathsetmacro{\px}{\cx+\couplerAlong*\qx-\couplerNormal*\qy}
            \pgfmathsetmacro{\py}{\cy+\couplerAlong*\qy+\couplerNormal*\qx}
            \xdef\fourbarCurveCoordinates{\fourbarCurveCoordinates (\px,\py)}
        }
        \ifnum\branch=-1
            \global\let\fourbarMinus\fourbarCurveCoordinates
        \else
            \global\let\fourbarPlus\fourbarCurveCoordinates
        \fi
    }

    \foreach \panel/\poseAngle in {0/71.565051177,1/150,2/210} {
        \begin{scope}[xshift={\panel*\panelseparation cm}]
            \draw[orange!85!black,thick] plot coordinates {\fourbarMinus};
            \draw[orange!85!black,thick] plot coordinates {\fourbarPlus};
            \ifnum\panel=0
                \def\cx{1}\def\cy{3}\def\dx{3}\def\dy{5}
            \else
                \pgfmathsetmacro{\cx}{sqrt(10)*cos(\poseAngle)}
                \pgfmathsetmacro{\cy}{sqrt(10)*sin(\poseAngle)}
                \pgfmathsetmacro{\distanceSquared}{26-8*\cx}
                \pgfmathsetmacro{\lambdaValue}{(\distanceSquared-18)/(2*\distanceSquared)}
                \pgfmathsetmacro{\muValue}{sqrt(8/\distanceSquared-(\lambdaValue)^2)}
                \pgfmathsetmacro{\dx}{\cx+\lambdaValue*(4-\cx)+\muValue*\cy}
                \pgfmathsetmacro{\dy}{\cy-\lambdaValue*\cy+\muValue*(4-\cx)}
            \fi
            \pgfmathsetmacro{\px}{\cx+\couplerAlong*(\dx-\cx)-\couplerNormal*(\dy-\cy)}
            \pgfmathsetmacro{\py}{\cy+\couplerAlong*(\dy-\cy)+\couplerNormal*(\dx-\cx)}
            \coordinate (A) at (0,0);
            \coordinate (B) at (4,0);
            \coordinate (C) at (\cx,\cy);
            \coordinate (D) at (\dx,\dy);
            \coordinate (P) at (\px,\py);
            \fill[blue!10,fill opacity=0.7] (C)--(D)--(P)--cycle;
            \draw[blue!60!black,thick] (C)--(P)--(D);
            \draw[dashed, thick] (A)--(B);
            \draw[thick] (A)--(C) (B)--(D);
            \draw[blue!60!black,very thick] (C)--(D);
            \ifnum\panel=2
                \node[pivot,label=above right:{}] at (A) {};
            \else
                \node[pivot,label=below left:{}] at (A) {};
            \fi
            \node[pivot,label=below right:{}] at (B) {};
            \ifnum\panel=2
                \node[pivot,label=right:{}] at (C) {};
                \node[pivot,label=above left:{}] at (D) {};
                \node[pivot,label=above left:{}] at (P) {};
            \else
                \ifnum\panel=1
                    \node[pivot,label=below:{}] at (C) {};
                \else
                    \node[pivot,label=left:{}] at (C) {};
                \fi
                \node[pivot,label=right:{}] at (D) {};
                \node[pivot,label=above:{}] at (P) {};
            \fi
        \end{scope}
    }
\end{tikzpicture}
    \caption{Three configurations of a four-bar mechanism and its coupler curve (orange).}
    \label{fig:coupler}
\end{figure}
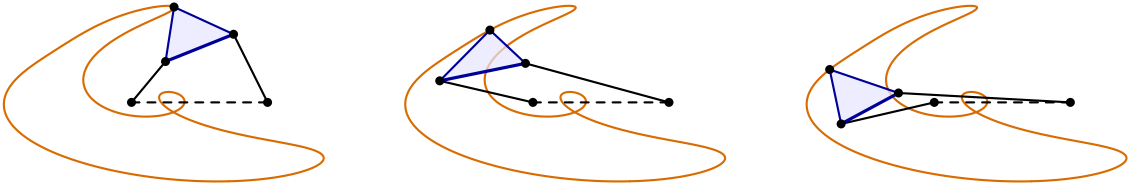
\vspace{-10pt}

A four-bar mechanism consists of four fixed-length edges, or \textit{links}, joined
cyclically, with one link held fixed. Such a configuration is not rigid; it may be configured in a continuum of ways and enjoys one-degree-of-freedom motion. Placing a rigid body, like a triangle, on the opposite edge of the ground link models this motion via the path traced out by the furthest vertex. This path is the coupler curve of a four-bar mechanism. Despite its simplicity as a $1$-DOF planar linkage, its interpolation problem is already very difficult. 

The conjectural answer of $1442$ to Alt's problem was first reported  in $1992$ by Morgan, Sommese, and Wampler \cite{MSW} in their celebrated paper, \textit{Complete Solution of the Nine-Point Path Synthesis Problem for Four-Bar Linkages}. By formulating the nine-point path synthesis problem in \textit{isotropic coordinates} (see \Cref{sec:Mechanisms}), they reduced the numerical start system from size $7^8=5,764,801$ to just $286,720$, placing the problem within range of homotopy continuation: they followed the start solutions to solve for the $8652$ labeled four-bar mechanisms which trace out the $1442$ coupler curve solutions.  Modern certification software like \cite{Certify,KisunCertification,alphaCertified} can \textit{prove} that the computed solutions correspond to true solutions, thus proving that $1442$ is a lower bound. Proving that there are no others has since been the content of Alt's problem. In the present work, we solve Alt's problem and prove that the numerically established enumeration of Morgan, Sommese, and Wampler is correct.
\begin{theorem}[Alt's Theorem]
\label{thm:alt}
There are $1442$ coupler curves through nine generic points in ${\mathbb{C}}^2$. 
\end{theorem}

Our approach, like that of \cite{MSW}, uses a thoughtful change of coordinates. While classical formulations solve for the  $8652$ labeled four-bar mechanisms which draw the $1442$ coupler curves, we instead solve for the coupler curves directly. We do this by providing a parametrization of the space $\mathcal C$ of coupler curves by the Grassmannian $\textrm{Gr}(3,6)$. It is helpful to think of $\textrm{Gr}(3,6)$ as the space of nets of conics. From this point of view, it becomes clear that $\dim(\mathcal C) = 9$. 

We formulate the nine interpolation conditions in an affine chart of $\textrm{Gr}(3,6)$ and produce an immediate reduction from $286,720$ as an upper bound for Alt's problem to just $14001$, given by a \textit{Kouchnirenko bound} (see \cite{Kouchnirenko1976}). Another linear-algebraic change of coordinates reduces this further to the \textit{BKK bound} of $6064$ (see \Cref{prop:BKK}). A final birational change of variables rewrites the coordinates as quotients of Pl\"ucker coordinates obtaining the final reduction to a BKK upper bound of $5538$. This is just $2^{12} = 4096$ larger than the hoped-for answer of Alt's problem. 

We have no explanation for why the deficit $4096$ is a power of two. Rather, we discover $4096$ as  
{
\newcommand{\tp}{\!+\!}
\begin{align*}
&4096=
\underbrace{2\cdot661}_{\mathcal B_1}
\tp\underbrace{1\cdot335}_{\mathcal B_2}
\tp\underbrace{\left(1\cdot1648\tp2\cdot271\right)}_{\mathcal B_3}
\tp\underbrace{\left(1\cdot6\tp2\cdot7\tp3\cdot18\tp5\cdot6\right)}_{\mathcal B_4}
\tp\underbrace{\left(3\cdot6\tp10\cdot2\tp6\cdot2\right)}_{\mathcal B_5}
\tp\underbrace{4\cdot4}_{\mathcal B_6}\\
&\tp\underbrace{4\cdot4}_{\mathcal B_7}
\tp\underbrace{2\cdot2}_{\mathcal B_8}
\tp\underbrace{1\cdot8}_{\mathcal B_9}
\tp\underbrace{4\cdot2}_{\mathcal B_{10}}
\tp\underbrace{\left(6\cdot1\tp10\cdot1\right)}_{\mathcal B_{11}}
\tp\underbrace{\left(3\cdot2\tp4\cdot1\right)}_{\mathcal B_{12}}
\tp\underbrace{3\cdot2}_{\mathcal B_{13}}
\tp\underbrace{4\cdot1}_{\mathcal B_{14}}
\tp\underbrace{1\cdot2}_{\mathcal B_{15}}
\tp\underbrace{1\cdot2}_{\mathcal B_{16}}
\tp\underbrace{1\cdot1}_{\mathcal B_{17}}
\tp\underbrace{1\cdot2}_{\mathcal B_{18}}.
\end{align*}
}
This partition is tropical and comes from a bookkeeping procedure attached to a sparse degeneration. Specifically, we form a straight-line homotopy in $\varepsilon$ from a generic polynomial system with the same support as ours to the system with $1442$ solutions. In the limit, $4096$ of those solutions leave the relevant, non-compact solution space. Each does so at some asymptotic rate in its coordinates with respect to $\varepsilon$, called the \textit{tropical vector} of that branch. Provably accounting for $4096$ branches via their tropical vectors produces the partition above: each term represents the contribution of one tropical solution via the product of its count and a multiplicity; terms are grouped if their solutions escape the solution space to the same boundary stratum. Provably performing this accounting requires meticulous bookkeeping.

The key in proving the numbers above lies in Bernstein's \textit{other theorem}, which establishes that the failure to reach the BKK bound lies in these tropical counts, and that such solutions must degenerate to boundary loci associated to \textit{facial systems} of the original equations. By analyzing these facial systems symbolically and introducing sufficiently many variables to track higher-order differential information about these paths, we construct square polynomial systems whose simple solutions are responsible for the reduction from the BKK bound. 
Certification plays two complementary roles in our work. Certifying
solutions of the interpolation system gives a lower bound on the number
of coupler curves. Certifying solutions of the 
systems describing branches of a degeneration with a particular tropical vector gives a lower bound on the number of branches which escape when the generic sparse system is degenerated to an instance of our formulation. Hence, these provide an upper bound on the number of coupler curves. Simple solutions
of these systems persist under small changes of the
parameters, so a single certified instance supplies the required lower
bounds for generic parameters. This hybrid certification procedure successfully closes the gap from $5538$ to $1442$ by accounting for all $4096$ escaping solution branches of the homotopy.

\subsection{AI Statement and Discussion}
Generative AI tools, principally OpenAI’s ChatGPT, were used during this project to assist with literature searches, exploratory calculations, the construction and debugging of computational scripts, and the editing of the manuscript. In particular, they assisted in organizing the boundary analysis, constructing the compatibility systems systematically, and implementing the numerical certification calculations. All mathematical arguments, computational formulations, and reported results were independently checked by the author, who takes full responsibility for the contents of the paper.

\subsection{Outline of paper} \Cref{sec:Mechanisms} introduces four-bar mechanisms, isotropic coordinates, and coupler curves.  \Cref{sec:NetsToCouplerAndBack} develops the correspondence between $\textrm{Gr}(3,6)$ and the space of coupler curves. \Cref{sec:SparseReduction} formulates the nine-point path synthesis problem in the coordinates of $\textrm{Gr}(3,6)$. It also modifies this polynomial system to reach a mixed volume of $5538$. \Cref{sec:TropicalBookKeeping} gives our main theoretical framework for certifying sparse degree-loss via tropical accounting. \Cref{sec:BoundaryDetails} provides the census of the partition of $4096$ via the tropical data of the polynomial system, and formulates the square systems whose numerical certifications provide lower bounds for these deficits, therefore producing upper bounds for the answer to Alt's problem. That section concludes with a proof of \Cref{thm:alt}.

\section*{Acknowledgements} The author was supported by NSERC discovery grant RGPIN-2023-03551.
 
\section{Isotropic coordinates, mechanisms, and coupler curves}
\label{sec:Mechanisms}
One of the main insights of this work is the birational identification of coupler curves with the space of nets of conics:
\[
\{\text{Coupler curves of four-bar mechanisms}\} \leftrightarrow \textrm{Gr}(3,\textrm{Sym}^2((\mathbb{C}^3)^*)) \cong \textrm{Gr}(3,6)
\]
The purpose of this section is to make that correspondence explicit. We then provide a formulation of the nine-point path synthesis problem in these coordinates.

\subsection{Isotropic coordinates}
Fix the complex plane $\mathbb{C}_{x,y}^2$ with coordinates $x,y$. \mydef{Isotropic coordinates} for $\mathbb{C}_{x,y}^2$ are given by the following complex linear change of coordinates:
\begin{equation}
\mydef{z} = x+iy \quad \quad \quad \quad \quad \quad \mydef{w} = x-iy.
\end{equation}
We denote the plane in isotropic coordinates as $\mathbb{C}_{z,w}^2$. Counterintuitively, perhaps, $z$ and $w$ need not be conjugate: this occurs precisely when $x$ and $y$ are real. Thus, the real plane $\mathbb{R}_{x,y}^2$
is identified with
\[
\mydef{\mathbb{R}^2_{\textrm{iso}}} = \{(z,w)\in\mathbb{C}_{z,w}^2  \mid  w=\overline z\}.
\]
A real vector $(x,y)$ can therefore be represented by the single
complex number $z=x+iy$, with its second isotropic coordinate
recovered as $w=\overline z$. In general, however, $z$ and $w$ are  independent. 

A major benefit of these coordinates is that the squared length
becomes the isotropic product:
\[
x^2+y^2=zw.
\]
Thus a vector with isotropic coordinates $(u,\widetilde u)$
has squared length $r$ precisely when $u\widetilde u=r$.
A sequence of $k$ vectors
with isotropic coordinates
$(u_1,\widetilde u_1),\ldots,(u_k,\widetilde u_k)$ forms a closed loop when
\begin{equation}
\label{eq:loops}
\sum_{j=1}^k u_j=0,\qquad
\sum_{j=1}^k \widetilde u_j=0.
\end{equation}
On the real locus these equations are conjugate, but 
over $\mathbb C$, they are independent conditions. See \cite{Wampler1996} for more background on isotropic coordinates. 

\subsection{Four-bar mechanisms} We give a very brief description of the anatomy of a four-bar mechanism. \Cref{fig:fourbar} shows a four-bar mechanism  with notation for its parts:
\begin{itemize}
\item $\mydef{A}$ and $\mydef{B}$ are the fixed pivots in $\mathbb{R}_{x,y}^2$. The \mydef{ground link} $AB$ never moves.
\item The points $\mydef{C}$ and $\mydef{D}$ are \mydef{moving pivots}. The \mydef{side links} are $AD$ and $BC$.
\item The side links are connected by the ground link and the \mydef{coupler link} $CD$. 
\item The coupler link is an edge of the \mydef{coupler triangle} $\Delta DpC$. The vertex $\mydef{p}$ is  the \mydef{coupler point}.
\item The lengths of the links are fixed: $
|AD|^2 = \mydef{r_1}, |DC|^2 = \mydef{r_2},$ and $|BC|^2 = \mydef{r_3}
$.
\item We assign notation to the isotropic coordinates of the links:
\[
(u_0,\widetilde{u}_0) = \overrightarrow{AB} 
\quad\quad (u_1,\widetilde{u}_1) = \overrightarrow{AD} \quad \quad
 (u_2, \widetilde{u}_2) = \overrightarrow{DC} \quad \quad (u_3, \widetilde{u}_3) = \overrightarrow{BC}
\]
\item The shape of the coupler triangle is specified by constants $\mydef{e_1},\mydef{e_2}$ through
\[
z(p)-z(D)=e_1\bigl(z(C)-z(D)\bigr),\qquad
w(p)-w(D)=e_2\bigl(w(C)-w(D)\bigr).
\]
\end{itemize}

\begin{figure}[!htpb]
\centering
\includegraphics{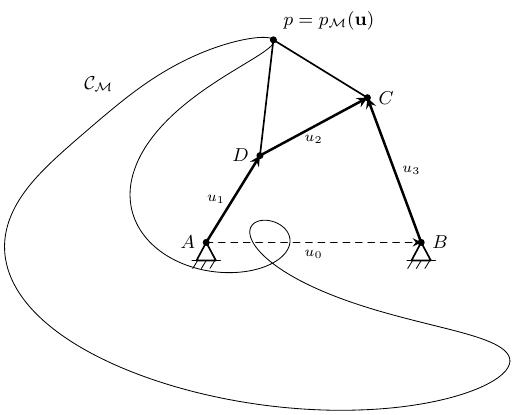}
\caption{The anatomy of a four-bar mechanism and notation for its parts.}
\label{fig:fourbar}
\end{figure}

Thus a (labeled) four-bar mechanism $\mydef{\mathcal M}$ is identified by the coordinates 
\[
\mathcal M \leftrightarrow (A,B,r_1,r_2,r_3,e_1,e_2)
\in \mydef{\mathbb{M}} = 
\underbrace{(\mathbb{C}_{z,w}^2)^2}_{\text{fixed pivots}}
\times
\underbrace{\mathbb{C}_r^3}_{\text{squared link lengths}}
\times
\underbrace{\mathbb{C}_e^2}_{\text{coupler triangle}}
\cong\mathbb{C}^9,
\]
and the coupler curve associated to $\mathcal M$ is denoted $\mydef{\mathcal C_\mathcal M}$.

The \mydef{fixed-length equations} are
\begin{equation}
\label{eq:lengths}
u_i\widetilde u_i=r_i,\quad i=1,2,3,
\end{equation}
and the \mydef{closure equations} are
\begin{equation}
\label{eq:linkclosureequations} 
u_1+u_2-u_3=u_0,\qquad
\widetilde u_1+\widetilde u_2-\widetilde u_3=\widetilde u_0.
\end{equation}
The coupler point $p = \mydef{p_{\mathcal M}(\textbf{u})}$ has isotropic coordinates depending on the mechanism $\mathcal M$ and the configuration $\mydef{\textbf{u}} = (u_1,\widetilde u_1, u_2, \widetilde u_2, u_3, \widetilde u_3)$:
\begin{equation}
\label{eq:couplerpointequation}
p_{\mathcal M}(\textbf{u}) = (z(p_{\mathcal M}(\textbf{u})),w(p_{\mathcal M}(\textbf{u}))) =(A_z+u_1+e_1u_2,A_w+\widetilde u_1+e_2\widetilde u_2).
\end{equation}
For a generic fixed mechanism $\mathcal M \in \mathbb{M}$, $u_0,\widetilde u_0$ are the constants $B_z-A_z$ and $B_w-A_w$, respectively, whereas
the configuration variables $\textbf{u}$ vary subject to these
constraints. The set of all tuples satisfying the fixed-length and closure conditions of a fixed mechanism $\mathcal M$ is the \mydef{configuration space} of $\mathcal M$, denoted $\mydef{E_{\mathcal M}}$. It is the vanishing of the ideal
\[
\mydef{\mathcal I_{\mathcal M}} \subseteq 
\mathbb{C}[u_1,\widetilde u_1, u_2, \widetilde u_2, u_3, \widetilde u_3] = \mathbb{C}[\textbf{u}]
\]
generated by the five equations \eqref{eq:lengths}-\eqref{eq:linkclosureequations} in the six \mydef{configuration variables} $\textbf{u}$. This configuration space determines the coupler curve $\mathcal C_{\mathcal M}$ precisely: it is the Zariski closure of $E_{\mathcal M}$
under the map 
\[
\textbf{u} \mapsto p_{\mathcal M}(\textbf{u}) =(A_z+u_1+e_1u_2,A_w+\widetilde u_1 + e_2 \widetilde u_2).
\]
 One can obtain a formula for its defining equation by eliminating the configuration variables $u_i,\widetilde u_i$ from the ideal describing the graph of $p_{\mathcal M}$:
\[
\mathcal I_\mathcal M + \langle A_z+u_1+e_1u_2-z,A_w+\widetilde u_1 + e_2 \widetilde u_2 -w \rangle \subseteq \mathbb{C}[\textbf{u},z,w].
\] This is the task of standard symbolic implicitization \cite[Chapter 3, \S3]{CoxLittleOShea2015}. The resulting ideal in $\mathbb{C}[z,w]$ is generated by a unique polynomial $\mydef{F_{\mathcal M}}$ up to scaling which cuts out the coupler curve $\mathcal C_{\mathcal M}$.  It has bidegree $(3,3)$, as we will also see from our formula for it in \Cref{sec:SparseReduction}. Thus, we write it as
\[
{F_{\mathcal M}(z,w)}=\sum_{i=0}^3\sum_{j=0}^3 \mydef{c_{ij}}z^iw^j.
\]
The sixteen coefficients determine a point
$[c_{00}:\cdots:c_{33}]\in\mathbb P^{15}$.
Let $\mydef{\mathcal C}$ be the Zariski closure of these coefficient
points as the mechanism varies. Thus $\mathcal C$ is the space of 
coupler-curve equations and their limits in $\mathbb{P}^{15}$. We identify $\mathcal C$ with the space of coupler curves and denote the rational map $\mathcal M \mapsto \mathcal C_{\mathcal M}$ by
\[
\mydef{\varphi}: \mathbb{M} \dashrightarrow \mathcal C.
\]

We define $\mydef{\mathcal E_{\mathbb{M}}}$ to be the space of mechanism-configuration pairs $(\mathcal M, \textbf{u})$, and $\mydef{\mathcal U_{\mathcal C}}$ to be the space of pairs $(C,p)$ consisting of a coupler curve $C \in \mathcal C$ and a point $p \in C$ on it.  We summarize the discussion in this section by the commutative diagram: 

\begin{figure}
\[
\begin{tikzcd}[row sep=large, column sep=large]
\mathcal E_{\mathbb M}
  \arrow[r, dashed] \arrow[d]
&
\mathcal U_{\mathcal C}
  \arrow[d]
\\
\mathbb M
  \arrow[r, dashed]
&
\mathcal C
\end{tikzcd}
\qquad\qquad
\begin{tikzcd}[row sep=large, column sep=large]
E_{\mathcal M}\ni\mathbf u
  \arrow[r, mapsto] \arrow[d]
&
p_{\mathcal M}(\mathbf u)\in\mathcal C_{\mathcal M}
  \arrow[d]
\\
\mathcal M
  \arrow[r, mapsto]
&
\mathcal C_{\mathcal M}
\end{tikzcd}
\]
\caption{A commutative diagram describing how the configuration space of a mechanism parametrizes a coupler curve.}
\label{fig:mechanism_commutative}
\end{figure}
where the top entries of the right diagram are fibres identified with their second factors. 
\begin{example}
\label{ex:hesse_mechanism}
Let $\zeta=(1+i\sqrt3)/2$. We take fixed pivots
$A=(1/2,\sqrt3/2)$ and $B=(-1,0)$ in Cartesian coordinates.
In isotropic coordinates, the mechanism data are
\[
\mathcal M=\bigl((\zeta,\overline\zeta),(-1,-1),2/3,2/3,2/3,\zeta,\overline\zeta\bigr).
\]
The ground link has length $\sqrt3$, all three moving links and the coupler triangle edges have
length $\sqrt{2/3}$. This mechanism and its coupler curve are displayed in \Cref{fig:hesse_mechanism}. We call this the \mydef{Hesse mechanism}, anticipating its
construction from a member of the Hesse pencil of cubics in
\Cref{sec:NetsToCouplerAndBack}.
\begin{figure}[!htpb]
\begin{center}
\includegraphics[scale=0.5]{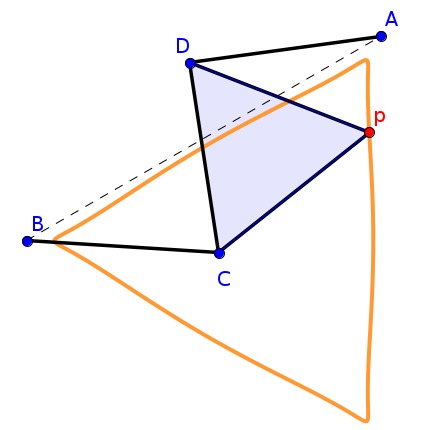}
\caption{The Hesse mechanism and its coupler curve.}
\label{fig:hesse_mechanism}
\end{center}
\end{figure}

Here $u_0=-1-\zeta$ and $\widetilde u_0=-1-\overline\zeta$.
The configuration ideal is therefore
\[
\mathcal I_{\mathcal M}
=
\left\langle
\begin{aligned}
&u_1\widetilde u_1-\tfrac23,\quad
 u_2\widetilde u_2-\tfrac23,\quad
 u_3\widetilde u_3-\tfrac23,\\
&u_1+u_2-u_3+1+\zeta,\quad
 \widetilde u_1+\widetilde u_2-\widetilde u_3+1+\overline\zeta
\end{aligned}
\right\rangle
\subseteq\mathbb C[\textbf{u}].
\]
The coupler point has isotropic coordinates
$p_{\mathcal M}(\textbf{u})=
(\zeta+u_1+\zeta u_2,\overline\zeta+\widetilde u_1+\overline\zeta\widetilde u_2)$.
Consequently, the ideal of its graph is
\[
\mathcal J_{\mathcal M}
=
\mathcal I_{\mathcal M}
+
\left\langle
\zeta+u_1+\zeta u_2-z,\,
\overline\zeta+\widetilde u_1+\overline\zeta\widetilde u_2-w
\right\rangle
\subseteq\mathbb C[\textbf{u},z,w].
\]
Eliminating the configuration variables gives
\[
\mathcal J_{\mathcal M}\cap\mathbb C[z,w]
=
\left\langle z^3w^3-2z^2w^2+4zw+z^3+w^3-1\right\rangle.
\]
Thus the coupler curve has Cartesian equation
\[
(x^2+y^2)^3-2(x^2+y^2)^2+4(x^2+y^2)+2x^3-6xy^2-1=0.
\]
\end{example}

\section{From nets of conics to coupler curves and back}
\label{sec:NetsToCouplerAndBack}
Consider  a three-dimensional space of conics in $\mathbb{P}^2$, called a \textit{net}. We explain how to find a distinguished basis (\Cref{lem:nicebasis}), construct a mechanism from an ordering of this basis (\Cref{thm:net_to_coupler}), and we prove  that the resulting coupler curve does not depend on that ordering. Conversely, we map  a mechanism to a net (\Cref{thm:mechanism_to_net}), whose associated coupler curve agrees with that of the original mechanism. Using the classical fact that a generic coupler curve is realized by six labeled four-bar mechanisms, we prove that the map from nets to coupler curves is generically injective (\Cref{cor:net_to_coupler_injective}).  It follows that all mechanisms realizing a generic coupler curve determine the same net, giving the inverse correspondence from coupler curves to nets. The main result of this section is that $\mathcal C$ is birational to $\textrm{Gr}(3,6)$ (\Cref{thm:coupler_to_net}). Finally, we argue that for nine generic points in the plane, none of the coupler curves through them lie outside of the locus where this correspondence is defined. 

\subsection{Nets to coupler curves}
In this section, we construct a coupler curve from a generic net
of conics. The idea is to construct, from such a net, an analogue of the configuration space $E_{\mathcal M}$ of the previous section.

Let $V = \mathbb{C}^3$ and consider the space of quadratic forms $\operatorname{Sym}^2(V^*) \cong \mathbb{C}^6$, identified with the affine cone over the five-dimensional projective space of conics. Let $\mydef{X}=(X_0,X_1,X_2)^T$ be homogeneous coordinates for $\mathbb P_{X}^2=\mathbb{P}_{\mathbb{C}}^2$ and consider an element
\[\mydef{\Lambda}=\langle \mydef{q_1},\mydef{q_2},\mydef{q_3} \rangle = \operatorname{span}\left\{\underbrace{X^T\mydef{M_1}X}_{q_1(X)},\underbrace{X^T\mydef{M_2}X}_{q_2(X)},\underbrace{X^T\mydef{M_3}X}_{q_3(X)}\right\} \in \textrm{Gr}(3,\operatorname{Sym}^2(V^*)) \cong \textrm{Gr}(3,6)
\] where the $M_i$ are symmetric $3 \times 3$ matrices. The space $\Lambda$ is a \mydef{net of conics} in $\mathbb P_{X}^2$. 

\begin{lemma}
\label{lem:nicebasis}
A generic net
$\Lambda\in\textrm{Gr}(3,\operatorname{Sym}^2(V^*))$
has a basis $\{Q_1,Q_2,Q_3\}$ of the form
\[
Q_i(X)=\ell_i(X)^2-\rho_iX_0^2,\qquad i=1,2,3,
\]
where each $\rho_i$ is nonzero,
$\{\ell_1,\ell_2,\ell_3\}$ is a basis of $V^*$, and $
\ell_1+\ell_2+\ell_3=X_0.$
\end{lemma}
\begin{proof}
Choose a basis $q_i(X)=X^TM_iX$ of $\Lambda$.
A conic $\sum_i\lambda_iq_i$ is singular at $P$ precisely
when $[M_1P\ M_2P\ M_3P]\lambda=0$, where
$\lambda=(\lambda_1,\lambda_2,\lambda_3)^T$.
Restricting $\det[M_1X\ M_2X\ M_3X]$ to $X_0=0$
gives a binary cubic.
For a generic net, it has three distinct roots $P_1,P_2,P_3$. At each root the matrix has rank two, selecting a unique
singular conic $Q_i$ up to scalar. Generically, each $Q_i$
has rank two and neither component is the line at infinity.
Its components are therefore distinct parallel affine lines,
so completing the square gives
$Q_i=\ell_i^2-\rho_iX_0^2$ with $\rho_i\ne0$.
The restrictions of the $\ell_i$ to $X_0=0$ have distinct
zeros, so their squares are linearly independent binary
quadratics. Thus the $Q_i$ form a basis of $\Lambda$.

The conditions above, together with independence of the
$\ell_i$, are nonempty open conditions: they hold for the
net generated by $X_1^2-\rho_1X_0^2$,
$X_2^2-\rho_2X_0^2$, and
$(X_0-X_1-X_2)^2-\rho_3X_0^2$, with nonzero $\rho_i$.
Indeed, its determinant on $X_0=0$ is proportional to
$X_1X_2(X_1+X_2)$. Now write $X_0=\sum_i c_i\ell_i$.
Each $c_i$ is nonzero, since otherwise restriction to
$X_0=0$ would give a dependence between two linear forms
with distinct zeros. Replacing $\ell_i,\rho_i,Q_i$ by
$c_i\ell_i,c_i^2\rho_i,c_i^2Q_i$, respectively, gives
the required normalization.
\end{proof}

\begin{example}
\label{ex:hesse_net}
We continue \Cref{ex:hesse_mechanism}. Consider the Hesse pencil
$G_t(X)=X_0^3+X_1^3+X_2^3-3tX_0X_1X_2$.
The partial derivatives of $G_2$ span the net
$\Lambda=\langle q_1,q_2,q_3\rangle$, where
\[
q_1=X_0^2-2X_1X_2,\qquad
q_2=X_1^2-2X_0X_2,\qquad
q_3=X_2^2-2X_0X_1.
\]
For the associated symmetric matrices, we obtain
\[
\det[M_1X\ M_2X\ M_3X]
=-(X_0^3+X_1^3+X_2^3+X_0X_1X_2),
\]
which defines another member of the Hesse pencil.
Writing $\zeta=(1+i\sqrt3)/2$ as before, its intersection with
$X_0=0$ consists of
$P_1=(0:1:-1)$, $P_2=(0:1:\overline\zeta)$, and
$P_3=(0:1:\zeta)$.
The construction in \Cref{lem:nicebasis} gives the normalized
linear forms
\[
\ell_1=\frac{X_0-X_1-X_2}{3},\qquad
\ell_2=\frac{X_0+\zeta X_1+\overline\zeta X_2}{3},\qquad
\ell_3=\frac{X_0+\overline\zeta X_1+\zeta X_2}{3}.
\]
These are linearly independent and satisfy
$\ell_1+\ell_2+\ell_3=X_0$.
Set $Q_i=\ell_i^2-\frac29X_0^2$. We obtain
\[
\begin{pmatrix}Q_1\\Q_2\\Q_3\end{pmatrix}
=
\frac19
\begin{pmatrix}
-1&1&1\\
-1&-\overline\zeta&-\zeta\\
-1&-\zeta&-\overline\zeta
\end{pmatrix}
\begin{pmatrix}q_1\\q_2\\q_3\end{pmatrix}.
\]
\begin{figure}[!htpb]
\begin{center}
\includegraphics[scale=0.45]{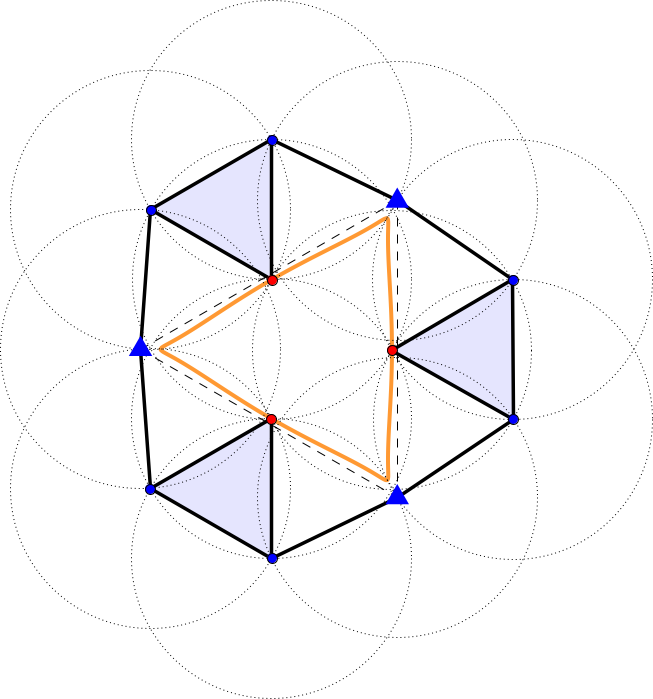}
\end{center}
\caption{The Roberts cognates of the Hesse mechanism. The circles indicate fixed leg lengths.}
\label{fig:cognates}
\end{figure}
Choosing a distinct ordering of this basis produces one of six distinct labeled mechanisms. The label swap symmetry produces three unlabeled mechanisms which draw the same coupler curve, called Roberts cognates (see \Cref{fig:cognates}). All three have $r_1=r_2=r_3=2/3$ and
$(e_1,e_2)=(\zeta,\overline\zeta)$.
They are related by rotations through $120^\circ$ about the
origin, which preserve the coupler curve. 
The other three orderings reverse the linkage labeling:
they exchange $A$ with $B$ and $D$ with $C$, replacing
$(e_1,e_2)$ by $(\overline\zeta,\zeta)$.
Thus there are six labeled mechanisms and three Roberts cognates.
\end{example}

Let $\mydef{Y}=(Y_0,Y_1,Y_2)^T$ be homogeneous coordinates for a second copy  $\mathbb P_{Y}^2=\mathbb{P}_{\mathbb{C}}^2$. We associate to the net $\Lambda$ the variety
\[
\mydef{E_\Lambda}
=\{([X],[Y])\in\mathbb P_{X}^2 \times  \mathbb P_{Y}^2 \mid 
X^TM_iY=0,\ i=1,2,3\},
\]
which is independent of the chosen basis of $\Lambda$. The variety $E_{\Lambda}$ will play the role of the configuration space $E_{\mathcal M}$. The bilinear forms associated to the basis $\{Q_i\}_{i=1}^3$ being zero will translate to the links of a mechanism constructed in the following theorem having constant length in isotropic coordinates.

\begin{theorem}[Nets to coupler curves]
\label{thm:net_to_coupler}
For a generic net $\Lambda$, each ordering of its distinguished basis from \Cref{lem:nicebasis} determines a labeled mechanism
$\mathcal M$ and an isomorphism
\[
E_\Lambda\cap\{X_0Y_0\ne0\}
\longrightarrow E_{\mathcal M}
\]
whose composition with $p_{\mathcal M}$ is
\[
([X],[Y])\longmapsto
\left(\frac{X_1}{X_0},\frac{Y_2}{Y_0}\right).
\]
Consequently, the Zariski closure of this image is a
four-bar coupler curve $\mydef{\mathcal C_\Lambda}$,
independent of the chosen ordering.
\end{theorem}
\begin{proof}
On the open subset of $E_{\Lambda}$ where $X_0Y_0 \neq 0$, set $\mydef{a_i} = \frac{\ell_i(X)}{X_0}$ and $\mydef{\widetilde a_i} = \frac{\ell_i(Y)}{Y_0}$. Observe that 
\[
a_i \widetilde a_i = \frac{\ell_i(X)\ell_i(Y)}{X_0Y_0} = \rho_i
\]
since $\ell_i(X)\ell_i(Y) -\rho_iX_0Y_0 = 0$ by the constraint that $([X],[Y]) \in E_{\Lambda}$. 

\Cref{lem:nicebasis} guarantees that $\{\ell_1,\ell_2,\ell_3\}$ is a basis for linear forms, and so we write $X_1$ and $X_2$ in terms of it:
\[
X_1 = \alpha_1\ell_1+\alpha_2\ell_2+\alpha_3\ell_3, \quad \quad X_2 = \beta_1 \ell_1 + \beta_2 \ell_2 + \beta_3 \ell_3
\]
and so the prescribed coordinate functions become 
\[
z = \frac{X_1}{X_0} = \alpha_1 a_1 + \alpha_2 a_2 +\alpha_3 a_3 \quad \quad w = \frac{Y_2}{Y_0} = \beta_1 \widetilde a_1 + \beta_2 \widetilde a_2 + \beta_3\widetilde a_3.
\] Equipped with this notation, we construct a mechanism
$\mathcal M$ and a map
\[
E_\Lambda\dashrightarrow E_{\mathcal M}
\]
whose composition with the coupler-point map is
$
([X],[Y])\longmapsto
\left(\frac{X_1}{X_0},\frac{Y_2}{Y_0}\right).
$
First, define the mechanism 
\[
\begin{aligned}
A&=(\alpha_3,\beta_3),&
B&=(\alpha_1,\beta_1),\\
u_0&=\alpha_1-\alpha_3,&
\widetilde u_0&=\beta_1-\beta_3,\\
r_i&=u_0\widetilde u_0\rho_i,\quad i=1,2,3,\\
e_1&=\frac{\alpha_2-\alpha_3}{u_0},&
e_2&=\frac{\beta_2-\beta_3}{\widetilde u_0}.
\end{aligned}
\]
We restrict to the open subset of nets where
$u_0\widetilde u_0\ne0$.
For $([X],[Y])\in E_\Lambda$ with $X_0Y_0\ne0$, define
its configuration coordinates by
\[
u_1=u_0a_1,\qquad \widetilde u_1=\widetilde u_0\widetilde a_1,\qquad u_2=u_0a_2,\qquad \widetilde u_2=\widetilde u_0\widetilde a_2,\qquad u_3=-u_0a_3,\qquad \widetilde u_3=-\widetilde u_0\widetilde a_3.
\]\Cref{lem:nicebasis} guarantees that  $\ell_1,\ell_2,\ell_3$ satisfy $
\ell_1(X)+\ell_2(X)+\ell_3(X) = X_0$, implying 
\[
a_1+a_2+a_3 = 1 \quad \quad \text{ and symmetrically,} \quad \quad \widetilde a_1+ \widetilde a_2+ \widetilde a_3 = 1.
\]
The identities $a_i\widetilde a_i=\rho_i$ give the
fixed-length equations, while
$\sum_i a_i=\sum_i\widetilde a_i=1$ gives the closure equations.
Moreover,
\[
\begin{aligned}
A_z+u_1+e_1u_2
&=\alpha_3+(\alpha_1-\alpha_3)a_1
             +(\alpha_2-\alpha_3)a_2
 =\frac{X_1}{X_0},\\
A_w+\widetilde u_1+e_2\widetilde u_2
&=\beta_3+(\beta_1-\beta_3)\widetilde a_1
            +(\beta_2-\beta_3)\widetilde a_2
 =\frac{Y_2}{Y_0}.
\end{aligned}
\]
Thus these formulas define the required map to
$E_{\mathcal M}$, including its coupler coordinates.

Conversely, a configuration determines
\[
(a_1,a_2,a_3)
=\frac{1}{u_0}(u_1,u_2,-u_3),
\qquad
(\widetilde a_1,\widetilde a_2,\widetilde a_3)
=\frac{1}{\widetilde u_0}
(\widetilde u_1,\widetilde u_2,-\widetilde u_3).
\]
Since the $\ell_i$ form a basis of $V^*$, these determine
unique vectors $X,Y$ with $\ell_i(X)=a_i$ and
$\ell_i(Y)=\widetilde a_i$. Closure gives $X_0=Y_0=1$,
and the fixed-length equations give $([X],[Y])\in E_\Lambda$.
This construction is a regular inverse, so the chart
$X_0Y_0\ne0$ of $E_\Lambda$ is isomorphic to $E_{\mathcal M}$.
\end{proof}
We denote the rational map sending a generic net $\Lambda$ to the coupler curve $\mathcal C_{\Lambda}$ by 
\[
\mydef{\Phi}: \textrm{Gr}(3,6) \dashrightarrow \mathcal C.
\]
By analogy with the space $\mathcal E_{\mathbb{M}}$ of configuration spaces of mechanisms, define 
\[
\mydef{\mathcal E_{\textrm{Gr}}}=\{(\Lambda,[X],[Y]) \mid ([X],[Y]) \in E_{\Lambda}\}.
\]
The commutative diagram summarizing this discussion is 
\begin{figure}
\[
\begin{tikzcd}[row sep=large, column sep=large]
\mathcal E_{\textrm{Gr}}
  \arrow[r, dashed] \arrow[d]
&
\mathcal U_{\mathcal C}
  \arrow[d]
\\
\textrm{Gr}(3,6)
  \arrow[r, dashed]
&
\mathcal C
\end{tikzcd}
\qquad
\begin{tikzcd}[row sep=large, column sep=large]
E_\Lambda\ni([X],[Y])
  \arrow[r, mapsto] \arrow[d]
&
\left(\frac{X_1}{X_0},\frac{Y_2}{Y_0}\right)
  \in\mathcal C_\Lambda
  \arrow[d]
\\
\Lambda
  \arrow[r, mapsto]
&
\mathcal C_\Lambda
\end{tikzcd}
\]
\caption{A commutative diagram describing how the configuration space of a net parametrizes a coupler curve.}
\label{fig:net_commutative}
\end{figure}
Here the right diagram describes a generic net $\Lambda$,
with the fibers identified with their second factors.
The top assignment is defined on $X_0Y_0\ne0$. Viewed in $\mathcal U_{\mathcal C}$,  its
image has Zariski closure $\{\mathcal C_\Lambda\} \times \mathcal C_{\Lambda}$.

\begin{remark}[Ordered bases of nets and Roberts cognates]
The mechanism $\mathcal M$ in \Cref{thm:net_to_coupler}
depends on the ordering of the distinguished conics
$Q_1,Q_2,Q_3$, but the coupler curve does not.
For a generic net, the six orderings produce six distinct
labeled mechanisms with the same coupler curve.
Exchanging $Q_1$ and $Q_3$ reverses the linkage labeling,
exchanging $A$ with $B$ and $D$ with $C$, but does not truly change the mechanism. Hence, there are three four-bar mechanisms which give the same coupler curve:  the three classical Roberts cognates \cite{Roberts1875}.
\end{remark}

\begin{remark}
Generically, $E_{\Lambda}$ is a genus-one cubic over $\mathbb{C}$, equipped with the fixed-point-free involution $[X] \leftrightarrow [Y]$, and the map $\left(\frac{X_1}{X_0}, \frac{Y_2}{Y_0} \right)$ has some nice properties with respect to this involution. The same is true for the configuration space $E_{\mathcal M}$, the involution $(u_i,\widetilde u_i)  \leftrightarrow \left(\frac{u_0}{\widetilde u_0}\widetilde u_i,\frac{\widetilde  u_0}{u_0} u_i \right)$, and the map $\textbf{u} \mapsto p_{\mathcal M}(\textbf{u})$ plays nicely with respect to it. This is a general phenomenon. We write the result below, but leave its proof to the appendix because ultimately it does not contribute to the validity of the proof of Alt's problem. 

\begin{lemma}
\label{lem:genusone}
Let $E$ be a smooth projective curve of genus one over $\mathbb C$,
$\tau$ a fixed-point-free involution, and $z,w$ rational functions
with pole divisors
\[
\mydef{D_z}=P_0+P_1+P_2,\qquad
\mydef{D_w}=\tau(P_0)+\tau(P_1)+\tau(P_2),
\]
where these six points are distinct.  Then the Zariski closure of the image of the map $p \mapsto (z(p),w(p))$ on $E$ is a coupler curve.
\end{lemma}
\end{remark}

\subsection{Coupler curves to nets}

\begin{theorem}[Mechanisms to nets]
\label{thm:mechanism_to_net}
For generic $\mathcal M \in \mathbb{M}$, write $\rho_i = \frac{r_i}{u_0\widetilde u_0}$ and define
\[
L=u_1+u_2-u_3,\qquad
Q_i=u_i^2-\rho_iL^2,
\quad i=1,2,3.
\]
In the coordinates
\[
X_0=L,\qquad
X_1=A_zL+u_0(u_1+e_1u_2),\qquad
X_2=A_wL+\widetilde u_0(u_1+e_2u_2),
\]
the span of these conics defines a net $\Lambda_{\mathcal M}$ such that $
\mathcal C_{\Lambda_{\mathcal M}}=\mathcal C_{\mathcal M}
$.
The assignment $\mathcal M\mapsto\Lambda_{\mathcal M}$
is a rational map $\mathbb M\dashrightarrow\textrm{Gr}(3,6)$.
\end{theorem}\begin{proof}
The bilinear form associated to $Q_i$ is
\[
u_i\widetilde u_i
-\rho_i(u_1+u_2-u_3)
(\widetilde u_1+\widetilde u_2-\widetilde u_3).
\]
On the locus satisfying the isotropic loop equations, this becomes
\[
u_i\widetilde u_i-\rho_i u_0\widetilde u_0
=u_i\widetilde u_i-r_i,
\]
and so the vanishing of the three bilinear
forms becomes equivalent to the fixed-length equations.
The conics are linearly independent: their restrictions
to $L=0$ are $u_1^2$, $u_2^2$, and $(u_1+u_2)^2$,
which are linearly independent binary quadratics.

The stated coordinate change as a matrix equation is
\[
\begin{pmatrix}
X_0\\X_1\\X_2
\end{pmatrix}
=
\begin{pmatrix}
1&1&-1\\
A_z+u_0&A_z+e_1u_0&-A_z\\
A_w+\widetilde u_0&A_w+e_2\widetilde u_0&-A_w
\end{pmatrix}
\begin{pmatrix}
u_1\\u_2\\u_3
\end{pmatrix}.
\]
 It has determinant
$u_0\widetilde u_0(e_1-e_2)$, so it is invertible
for a generic mechanism.
Apply the same  change to the $\widetilde{\cdot}$ variables
to define $Y$.
On the loop equation locus, $X_0=u_0$ and
$Y_0=\widetilde u_0$, so
\[
\frac{X_1}{X_0}=A_z+u_1+e_1u_2,\qquad
\frac{Y_2}{Y_0}=A_w+\widetilde u_1+e_2\widetilde u_2.
\]These are the coupler-point equations, so the configuration
$\mathbf u\in E_{\mathcal M}$ produces the same coupler point
as the corresponding pair
$([X],[Y])\in E_{\Lambda_{\mathcal M}}$.
Conversely, for any such pair with $X_0Y_0\ne0$, apply the
inverse coordinate change and choose the unique representatives
satisfying the isotropic loop equations
\[
u_1+u_2-u_3=u_0,\qquad
\widetilde u_1+\widetilde u_2-\widetilde u_3
=\widetilde u_0.
\]
Such representatives exist because $X_0Y_0\ne0$.
The bilinear equations then give the fixed-length equations.
Thus the two images coincide, and taking Zariski closures gives
$\mathcal C_{\Lambda_{\mathcal M}}=\mathcal C_{\mathcal M}$.
Finally, the conic coefficients and the inverse coordinate
change depend rationally on $\mathcal M$, proving rationality
of the assignment.
\end{proof}
We denote the map sending a mechanism $\mathcal M$ to a net $\Lambda_{\mathcal M}$ by
\[
\mydef{\Psi}: \mathbb{M} \dashrightarrow \textrm{Gr}(3,6).
\] 
In \Cref{thm:net_to_coupler}, we showed that all six of the orderings of the distinguished basis for $\Lambda$ produce a mechanism with a common coupler curve. Now, we show that any of those mechanisms gets mapped to  $\Lambda$ under $\Psi$. 
\begin{lemma}[Recovery of a net]
\label{lem:recover_net}
Let $\Lambda$ be a generic net, and let
$\mathcal M_1,\ldots,\mathcal M_6$ be the labeled mechanisms
obtained from the six orderings of its distinguished basis as in the proof of \Cref{thm:net_to_coupler}.
Then
\[
\Psi(\mathcal M_i)=\Lambda,\qquad i=1,\ldots,6.
\]
Moreover, these six labeled mechanisms are distinct.
\end{lemma}

\begin{proof}
Fix an ordering of the normalized distinguished basis
\[
Q_i(X)=\ell_i(X)^2-\rho_iX_0^2,\qquad
\ell_1+\ell_2+\ell_3=X_0,
\]
and let $\mathcal M$ be the mechanism reconstructed from this
ordering. Under the substitution
\[
u_1=u_0\ell_1(X),\qquad
u_2=u_0\ell_2(X),\qquad
u_3=-u_0\ell_3(X),
\]
we have
\[
L=u_1+u_2-u_3=u_0X_0.
\]
Since $r_i=u_0\widetilde u_0\rho_i$, the conics defining
$\Psi(\mathcal M)$ become
\[
u_i^2-\frac{r_i}{u_0\widetilde u_0}L^2
=u_0^2Q_i(X),\qquad i=1,2,3.
\]
Thus they span the original net.

The coordinate change defining $\Psi$ also recovers the original
projective coordinates:
\[
\begin{aligned}
L&=u_0X_0,\\
A_zL+u_0(u_1+e_1u_2)
&=u_0\bigl(A_zX_0+u_0(\ell_1+e_1\ell_2)\bigr)
=u_0X_1,\\
A_wL+\widetilde u_0(u_1+e_2u_2)
&=u_0\bigl(A_wX_0+\widetilde u_0(\ell_1+e_2\ell_2)\bigr)
=u_0X_2.
\end{aligned}
\]
Therefore $\Psi(\mathcal M)=\Lambda$. The same argument applies
to every ordering.

To see that the resulting mechanisms are distinct, write
\[
X_1=\sum_i\alpha_i\ell_i,\qquad
X_2=\sum_i\beta_i\ell_i.
\]
The three points $(\alpha_i,\beta_i)$ are noncollinear because
$X_0,X_1,X_2$ form a basis of $V^*$. The six orderings select
the six ordered pairs of these points as the fixed pivots
$(A,B)$, and hence produce six distinct labeled mechanisms.
\end{proof}

\begin{corollary}
\label{cor:net_to_coupler_injective}
The map
\[
\Phi:\textrm{Gr}(3,6)\dashrightarrow\mathcal C,
\qquad
\Lambda\longmapsto\mathcal C_\Lambda,
\]
is generically injective.
\end{corollary}

\begin{proof}
Suppose that two generic nets $\Lambda$ and $\Lambda'$ produce
the same coupler curve. By \Cref{lem:recover_net}, each net
produces six distinct labeled mechanisms, and a mechanism
constructed from both nets would satisfy
\[
\Psi(\mathcal M)=\Lambda
\qquad\text{and}\qquad
\Psi(\mathcal M)=\Lambda'.
\]
Thus, if $\Lambda\ne\Lambda'$, the two sets of six mechanisms
are disjoint. The common coupler curve would then be realized
by at least twelve labeled mechanisms. This contradicts the
classical Roberts cognate theorem, according to which a generic
coupler curve is realized by exactly six labeled mechanisms \cite{Roberts1875,ShermanHauensteinWampler2022}.
Hence $\Lambda=\Lambda'$.
\end{proof}

\begin{theorem}[Coupler curves to nets]
\label{thm:coupler_to_net}
The rational map
\[
\Phi:\textrm{Gr}(3,6)\dashrightarrow\mathcal C,
\qquad
\Lambda\longmapsto\mathcal C_\Lambda,
\]
is birational. Its rational inverse is the well-defined map
\[
\mathcal C_{\mathcal M}\longmapsto\Lambda_{\mathcal M}.
\]
\end{theorem}

\begin{proof}
By \Cref{thm:mechanism_to_net}, $
\Phi\circ\Psi=\varphi.$
Since $\mathcal C$ is the closure of the image of $\varphi$,
this identity shows that $\Phi$ is dominant.
By \Cref{cor:net_to_coupler_injective}, it is generically
injective, and hence birational.

The same identity gives
\[
\Phi(\Lambda_{\mathcal M})=\mathcal C_{\mathcal M}.
\]
Consequently, the rational inverse of $\Phi$ sends
$\mathcal C_{\mathcal M}$ to $\Lambda_{\mathcal M}$.
Its uniqueness shows that this net is independent of the
chosen mechanism.
\end{proof}
\begin{lemma}
\label{lem:general_points_avoid_bad_curves}
Let $\mathcal C^\circ\subseteq\mathcal C$ be any nonempty Zariski
open subset. Every coupler curve through nine generic points
belongs to $\mathcal C^\circ$.
\end{lemma}

\begin{proof}
Let $Z=\mathcal C\setminus\mathcal C^\circ$. Since
$\mathcal C$ is irreducible of dimension nine and $Z$ is
proper, we have $\dim Z\leq 8$. Consider the incidence variety
\[
\mathcal J
=
\{(C,p_1,\ldots,p_9)\in
Z\times(\mathbb C_{z,w}^2)^9
\mid p_i\in C,\ i=1,\ldots,9\}.
\]
For each $C\in Z$, the nine points vary independently on $C$,
so the fiber over $C$ has dimension at most nine. Therefore, by the theorem on the dimension of fibres \cite{Shafarevich}, we have $
\dim\mathcal J\leq\dim Z+9\leq17.$
Hence, the
projection of $\mathcal J$ cannot be dense in an $18$-dimensional space of nine points. Thus nine
generic points do not lie on any curve in $Z$.
\end{proof}
The birational correspondence 
\[
\textrm{Gr}(3,6) \mathrel{\dashleftarrow \dashrightarrow} \mathcal C
\]
and \Cref{lem:general_points_avoid_bad_curves} allow us to formulate the nine-point path synthesis problem in  the Grassmannian. 

The following commutative diagram summarizes the correspondence.
The bottom row sends a mechanism to its associated net and then
to its coupler curve. The top row also records a configuration
and its resulting coupler point, while the vertical maps forget
these additional data. Generically, $\Psi$ has degree six and
$\Phi$ is birational.
\[
\begin{tikzcd}[row sep=large, column sep=large, scale=1.5, transform shape]
\mathcal E_{\mathbb M}
  \arrow[r, dashed] \arrow[d]
&
\mathcal E_{\textrm{Gr}}
  \arrow[r, dashed] \arrow[d]
&
\mathcal U_{\mathcal C}
  \arrow[d]
\\
\mathbb M
  \arrow[r, dashed, "\Psi", "{6:1}"']
  \arrow[rr, dashed, bend right=45,
    "{\varphi}"']
&
\textrm{Gr}(3,6)
  \arrow[r, dashed, "\Phi", "1:1"']
&
\mathcal C
\end{tikzcd}
\]

\section{A sparse formulation of the nine-point path synthesis problem}
\label{sec:SparseReduction}
Now that we have constructed the \textit{right} solution space for the nine-point path synthesis problem, we reformulate the interpolation conditions in this space. First, we define the problem:
\begin{center}
\textbf{Nine-point path synthesis problem:} Given nine \mydef{precision points} with isotropic coordinates \[\mydef{p_1} = (z_1,w_1),\ldots,\mydef{p_9} =(z_9,w_9),\] find all coupler curves which interpolate all of them. 
\end{center}
These nine conditions on the nine-dimensional space of coupler curves can be written down in a Vandermonde form once we obtain a formula for the coupler curve of a net. 

\subsection{The coupler equation of a net} Recall the correspondence $\Phi$ which maps a net $\Lambda \in \textrm{Gr}(3,6)$ to its corresponding coupler curve $\mathcal C_{\Lambda}$ in accordance with its configuration space $E_{\Lambda}$:
\[
E_{\Lambda} \ni ([X],[Y]) \mapsto \left( \frac{X_1}{X_0},\frac{Y_2}{Y_0}\right) \in \mathcal C_{\Lambda}
\]
It turns out that it is not particularly hard to write down the coupler equation $F_{\Lambda}(z,w)$ in terms of $\Lambda$. Since $z = \frac{X_1}{X_0}$, we replace $X_1$ with $zX_0$. Similarly, we replace $Y_2$ with $wY_0$. Then the quadratic  $X^TMY$ for a generic $3 \times 3$ symmetric matrix $M=(m_{i,j})_{i,j=0}^2$ becomes 
\begin{align}
\notag (X_0,zX_0,X_2) M \begin{pmatrix} Y_0 \\ Y_1 \\ w Y_0 \end{pmatrix}=&\underbrace{m_{1,2}}_{\mydef{b_{2,1}}}X_{2}Y_{1}+\left(\underbrace{m_{1,1}z+m_{0,1}}_{\mydef{b_{0,1}}}\right)X_{0}Y_{1}+\left(\underbrace{m_{2,2}w+m_{0,2}}_{\mydef{b_{2,0}}}\right)X_{2}Y_{0} \\
\notag &+\left(\underbrace{m_{1,2}z\,w+m_{0,1}z+m_{0,2}w+m_{0,0}}_{\mydef{b_{0,0}}}\right)X_{0}Y_{0} \\
&= \begin{pmatrix} b_{2,1} & b_{0,1} & b_{2,0} & b_{0,0} \end{pmatrix} \begin{pmatrix} X_2Y_1 \\ X_0Y_1 \\ X_2Y_0 \\ X_0Y_0 \end{pmatrix}
\end{align}
Now, given a net of conics  $
\Lambda = \textrm{span}\left(X^TM_1X,X^TM_2X,X^TM_3X\right)
$, we vectorize $M_1,M_2$ and $M_3$ via an ordered basis of monomials and express $\Lambda$ as the row space of the matrix \[
\mydef{\textbf{M}_{\Lambda}}=
\begin{pNiceMatrix}[first-row]
X_0^2 & 2X_0X_1 & 2X_0X_2 & X_1^2 & 2X_1X_2 & X_2^2\\
\vert & \vert & \vert & \vert & \vert & \vert\\
\mydef{\mathbf c_{0,0}} & \mydef{\mathbf c_{0,1}} & \mydef{\mathbf c_{0,2}} &
\mydef{\mathbf c_{1,1}} & \mydef{\mathbf c_{1,2}} & \mydef{\mathbf c_{2,2}}\\
\vert & \vert & \vert & \vert & \vert & \vert
\end{pNiceMatrix}.
\]
We extend the notation $b_{i,j}(m_{0,0},\ldots,m_{2,2})$ to
 $\mydef{\mathbf{b_{i,j}}} = b_{i,j}
 (\mathbf{c}_{0,0},\ldots,\mathbf{c}_{2,2})$ which takes values in $\mathbb{C}^3$. The condition now for $([X],[Y]) \in E_{\Lambda}$ to parametrize $(z,w)$ is 
\[\underbrace{\begin{pmatrix} \mathbf b_{2,1} &\mathbf b_{0,1} & \mathbf b_{2,0} &\mathbf  b_{0,0} \end{pmatrix}}_{\mydef{\mathbf{b}(\mathbf{c})}} \begin{pmatrix} X_2Y_1 \\ X_0Y_1 \\ X_2Y_0 \\ X_0Y_0 \end{pmatrix} = \textbf{0}
\]
Writing $\mydef{[i,j,k]}=\textrm{det}(i,j,k)$, the signed maximal minors
\[
\mydef{k_{2,1}}=-[\mathbf b_{0,1},\mathbf b_{2,0},\mathbf b_{0,0}], \,\, 
\mydef{k_{0,1}}=[\mathbf b_{2,1},\mathbf b_{2,0},\mathbf b_{0,0}],  \,\, 
\mydef{k_{2,0}}=-[\mathbf b_{2,1},\mathbf b_{0,1},\mathbf b_{0,0}], \,\, 
\mydef{k_{0,0}}=[\mathbf b_{2,1},\mathbf b_{0,1},\mathbf b_{2,0}]
\]
 of the matrix $\mathbf{b}(\mathbf{c})$ give a kernel vector $\mydef{k}=(k_{2,1},k_{0,1},k_{2,0},k_{0,0})^T$ of it. For $k$ to be a kernel vector of the form $(X_2Y_1, X_0Y_1, X_2Y_0, X_0Y_0)$, it must satisfy the associated binomial 
 \[
 k_{2,1}k_{0,0}-k_{2,0}k_{0,1} =0 \longleftrightarrow (X_2Y_1)(X_0Y_0)=(X_2Y_0)(X_0Y_1).
 \]
 What we have just accomplished is the derivation of a formula for the coupler curve of a net. 
 \begin{theorem}
 Let $\Lambda$ be a generic net. Its coupler curve has the equation 
 \[
 F_{\Lambda}(z,w) =  k_{2,1}k_{0,0}-k_{2,0}k_{0,1}=\sum_{i,j=0}^3 c_{ij}(\Lambda)z^iw^j=0.
\]
Each coefficient $c_{i,j}$ is quadratic in the Pl\"ucker coordinates of $\textbf{M}_{\Lambda}$. In particular, 
\[
c_{00}(\Lambda) = p_{125}p_{135}-p_{123}p_{235} \quad \quad c_{33}(\Lambda) = -p_{456}^2.
\]
 \end{theorem}
 \begin{proof}
 We have already established that $k_{2,1}k_{0,0}-k_{2,0}k_{0,1}$ is the coupler equation. The four minors have bidegrees at most $(2,2),(2,1),(1,2),$ and $(1,1)$, respectively, so the coupler curve has bidegree at most $(3,3)$. A rational evaluation suffices to show that this is generic behavior. 
 The claims about  $c_{00}(\Lambda)$ and $c_{33}(\Lambda)$ can be verified computationally. 
 \end{proof}
 \subsection{The nine-point equations}
 Fix nine precision points in their isotropic coordinates
 \[
 p_1 = (z_1,w_1), \quad p_2 = (z_2,w_2), \quad \cdots \quad p_9=(z_9,w_9).
 \]
 The condition that $p_i \in \mathcal C_\Lambda$ translates to $F_{\Lambda}(z_i,w_i) = 0$. We fix the chart $p_{456}=1$ on the Grassmannian by setting the rightmost block of $\textbf{M}_{\Lambda}$ equal to the identity:
\[
\mydef{\textbf{M}(\textbf{t})} = 
\begin{pmatrix}
t_0&t_3&t_6&1&0&0\\
t_1&t_4&t_7&0&1&0\\
t_2&t_5&t_8&0&0&1
\end{pmatrix}.
\]
Since $c_{33}(\Lambda) = -p_{456}^2$, this fixes a coordinate chart $c_{33} = -1$ on the projective space of coupler curves. The matrix $\textbf{b}$ in the variables $\textbf{t}$ is 
\[
\textbf{b}(\textbf{t})=\left(\!\begin{array}{cccc}
      0&z+t_{3}&t_{6}&t_{3}z+t_{6}w+t_{0}\\
      1&t_{4}&t_{7}&z\,w+t_{4}z+t_{7}w+t_{1}\\
      0&t_{5}&w+t_{8}&t_{5}z+t_{8}w+t_{2}
      \end{array}\!\right)
\]
Its kernel  is spanned by 
$k(\textbf{t})=(k_{21},k_{01},k_{20},k_{00})^T$ where 
 
 {\footnotesize{
\begin{align*}
k_{21} &= -z\,w-t_{8}z-t_{3}w+t_{5}t_{6}-t_{3}t_{8}\\
k_{01} &=t_{5}z^{2}+t_{8}z\,w+t_{2}z+\left(-t_{5}t_{6}+t_{3}t_{8}\right)w+t_{2}t_{3}-t_{0}t_{5}\\
k_{20} &=t_{3}z\,w+t_{6}w^{2}+\left(-t_{5}t_{6}+t_{3}t_{8}\right)z+t_{0}w-t_{2}t_{6}+t_{0}t_{8}\\
k_{00} &=z^{2}w^{2}+\left(t_{4}+t_{8}\right)z^{2}w+\left(t_{3}+t_{7}\right)z\,w^{2}+\left(-t_{5}t_{7}+t_{4}t_{8}\right)z^{2}+\left(-t_{5}t_{6}+t_{3}t_{8}+t_{1}\right)z\,w+\left(-t_{4}t_{6}+t_{3}t_{7}\right)w^{2}\\
&+\left(-t_{2}t_{7}+t_{1}t_{8}\right)z+\left(t_{1}t_{3}-t_{0}t_{4}\right)w+t_{2}t_{4}t_{6}-t_{1}t_{5}t_{6}-t_{2}t_{3}t_{7}+t_{0}t_{5}t_{7}+t_{1}t_{3}t_{8}-t_{0}t _{4}t_{8}
\end{align*}
}}
\noindent Taking $k_{21}k_{00}-k_{01}k_{20}= F_{\mathbf t}(z,w)$ then gives a formulation of the \mydef{nine-point path synthesis equations}:
\[
\mathcal F^{(0)}(\textbf{t};\textbf{z},\textbf{w}) = \begin{pmatrix}
F_{\mathbf{t}}(z_1,w_1) \\ 
F_{\mathbf{t}}(z_2,w_2) \\ 
\vdots \\ 
F_{\mathbf{t}}(z_9,w_9) \\ 
\end{pmatrix} = \textbf{0}
\]
Each polynomial in $t_0,\ldots,t_8$ has the same support, with
Newton polytope $\Delta^{(0)}$ of normalized volume
\[
9!\operatorname{vol}(\Delta^{(0)})
=
9!\frac{4667}{120960}
=
14001.
\]
Since all of the Newton polytopes are the same, this is a \textit{Kouchnirenko bound} \cite{Kouchnirenko1976}.
This volume computation and all subsequent mixed volume computations were first computed with PHCpack \cite{Verschelde1999} and subsequently with \cite{MixedSubdivisions2019}. 
Consequently, the mixed volume of the original system is $14001$.
The vertices of $\Delta^{(0)}$, grouped according to total degree
$0,1,2,3,4,5$, are given below.
{\tiny
\begin{align*}
&V_0 = \left(\begin{array}{c}
0\\
0\\
0\\
0\\
0\\
0\\
0\\
0\\
0
\end{array}\right),\qquad
V_1 = \left(\begin{array}{ccc}
0&0&0\\
0&0&1\\
0&0&0\\
0&0&0\\
0&1&0\\
0&0&0\\
0&0&0\\
1&0&0\\
0&0&0
\end{array}\right), V_2 = 
\left(\begin{array}{ccccccccccccccc}
0&0&0&0&0&0&0&0&0&0&0&1&1&1&1\\
0&0&0&0&0&0&0&0&0&0&0&0&0&0&0\\
0&0&0&0&0&0&0&0&1&1&1&0&0&0&1\\
0&0&0&0&0&1&1&2&0&0&1&0&0&0&0\\
0&0&0&0&1&0&1&0&0&0&0&0&0&1&0\\
0&0&0&1&0&1&0&0&0&0&0&0&1&0&0\\
0&0&1&0&1&0&0&0&0&1&0&0&0&0&0\\
0&1&0&1&0&0&0&0&1&0&0&0&0&0&0\\
2&1&1&0&0&0&0&0&0&0&0&1&0&0&0
\end{array}\right),\end{align*}
\begin{align*}
&V_3 = \left(\begin{array}{cccccccccccccccccccccccccccccc}
0&0&0&0&0&0&0&0&0&0&0&0&0&0&0&0&0&0&0&0&0&1&1&1&1&1&1&1&1&2\\
0&0&0&0&0&0&0&0&0&0&0&0&0&0&0&0&0&0&0&1&1&0&0&0&0&0&0&0&0&0\\
0&0&0&0&0&0&0&0&0&0&0&0&1&1&1&1&1&1&2&0&0&0&0&0&0&0&0&1&1&0\\
0&0&0&0&0&0&0&1&1&1&1&2&0&0&0&1&1&2&0&0&2&0&0&0&1&1&1&0&1&0\\
0&0&0&0&1&1&1&0&0&0&1&0&0&0&1&0&0&0&0&0&0&0&0&0&0&0&1&0&0&0\\
1&1&1&2&0&0&1&0&1&1&0&0&0&0&0&0&0&0&0&0&0&0&1&1&0&1&0&0&0&1\\
0&1&2&1&0&1&1&1&0&0&1&0&0&1&1&0&1&0&1&0&0&0&0&0&0&0&0&0&0&0\\
1&1&0&0&0&0&0&0&0&1&0&1&1&0&0&0&0&0&0&0&0&0&0&1&0&0&0&0&0&0\\
1&0&0&0&2&1&0&1&1&0&0&0&1&1&0&1&0&0&0&2&0&2&1&0&1&0&0&1&0&0
\end{array}\right),\\[4pt]
&
V_4 = \left(\begin{array}{cccccccccccccccccccccccccccccc}
0&0&0&0&0&0&0&0&0&0&0&0&0&0&0&0&0&0&0&1&1&1&1&1&1&1&1&1&1&2\\
0&0&0&0&0&0&0&0&0&0&0&0&0&0&0&0&0&0&0&0&0&0&0&0&0&0&0&0&0&0\\
0&0&0&0&0&0&0&0&0&0&1&1&1&1&1&1&1&1&2&0&0&0&0&0&0&0&0&1&1&0\\
0&0&0&0&1&1&1&1&2&2&0&0&0&1&1&1&2&2&1&0&0&0&0&0&1&1&1&0&1&0\\
0&0&1&1&0&0&1&1&0&0&0&0&1&0&0&1&0&0&0&0&0&0&1&1&0&0&0&0&0&0\\
2&2&1&1&1&1&0&0&0&0&1&1&0&0&1&0&0&0&0&1&1&2&0&1&0&1&1&1&0&1\\
1&2&1&2&0&1&0&1&0&0&1&2&1&1&1&1&0&0&1&0&1&1&0&1&0&0&0&1&0&0\\
1&0&0&0&1&1&0&0&0&1&1&0&0&0&0&0&0&1&0&1&0&0&0&0&0&0&1&0&0&0\\
0&0&1&0&1&0&2&1&2&1&0&0&1&1&0&0&1&0&0&1&1&0&2&0&2&1&0&0&1&1
\end{array}\right),\\[4pt]
&
V_5 = \left(\begin{array}{cccccccccc}
0&0&0&0&0&0&1&1&1&1\\
0&0&0&0&1&1&0&0&0&0\\
1&1&1&1&0&0&0&0&0&0\\
0&1&1&2&0&2&0&0&1&1\\
1&0&1&0&0&0&0&1&0&1\\
1&1&0&0&2&0&2&1&1&0\\
2&1&1&0&2&0&1&1&0&0\\
0&1&0&1&0&0&1&0&1&0\\
0&0&1&1&0&2&0&1&1&2
\end{array}\right).
\end{align*}
}
\noindent The degree-five homogeneous part of every point condition is the
same polynomial:
\[
\left(t_5t_6-t_3t_8\right)
\left(
t_2t_4t_6-t_1t_5t_6-t_2t_3t_7
+t_0t_5t_7+t_1t_3t_8-t_0t_4t_8
\right)=-p_{123}p_{235}
\]
In particular, this part is independent of the precision point
$(z_i,w_i)$. Subtracting the first point condition from each of
the remaining eight therefore eliminates every monomial of total
degree five from those equations. We obtain the equivalent system
\[
\mydef{\mathcal F^{(1)}(\textbf{t};\textbf{z},\textbf{w})}
=
\begin{pmatrix}
F_{\mathbf t}(z_1,w_1)\\
F_{\mathbf t}(z_2,w_2)-F_{\mathbf t}(z_1,w_1)\\
\vdots\\
F_{\mathbf t}(z_9,w_9)-F_{\mathbf t}(z_1,w_1)
\end{pmatrix}
=\mathbf 0.
\]
The first equation still has degree five, while the remaining
eight have degree four. This elementary change of
generators lowers the mixed volume  $
14001 \to 12597.$ This is one linear change of generators in the polynomial ring $\mathbb{C}(z_1,w_1,\ldots,z_9,w_9)[t_0,\ldots,t_8]$, but there are much better ones informed by basic interpolation. 
\subsection{Vandermonde reduction}
Consider the $9\times16$ multivariate Vandermonde matrix, in this particular order
{\footnotesize
\[
\mydef{\mathbb V}
=
\begin{pNiceMatrix}[first-row]
c_{00}&c_{01}&c_{02}&c_{03}&
c_{10}&c_{11}&c_{12}&c_{13}&
c_{20}&c_{21}&c_{22}&c_{23}&
c_{30}&c_{31}&c_{32}&c_{33}\\ \hline
1&w_1&w_1^2&w_1^3&
z_1&z_1w_1&z_1w_1^2&z_1w_1^3&
z_1^2&z_1^2w_1&z_1^2w_1^2&z_1^2w_1^3&
z_1^3&z_1^3w_1&z_1^3w_1^2&z_1^3w_1^3\\
1&w_2&w_2^2&w_2^3&
z_2&z_2w_2&z_2w_2^2&z_2w_2^3&
z_2^2&z_2^2w_2&z_2^2w_2^2&z_2^2w_2^3&
z_2^3&z_2^3w_2&z_2^3w_2^2&z_2^3w_2^3\\
\vdots&\vdots&\vdots&\vdots&
\vdots&\vdots&\vdots&\vdots&
\vdots&\vdots&\vdots&\vdots&
\vdots&\vdots&\vdots&\vdots\\
1&w_9&w_9^2&w_9^3&
z_9&z_9w_9&z_9w_9^2&z_9w_9^3&
z_9^2&z_9^2w_9&z_9^2w_9^2&z_9^2w_9^3&
z_9^3&z_9^3w_9&z_9^3w_9^2&z_9^3w_9^3
\end{pNiceMatrix}.
\]
}
The nine-point system is simply $
\mathcal F=\mathbb V\mathbf c(\mathbf t)=\mathbf 0$, where 
$\mydef{
\textbf{c}(\textbf{t})} = (c_{00}(\textbf{t}),\ldots,c_{33}(\textbf{t}))^T
$
are the coefficients of the defining equation of the coupler curve associated to $\textbf{M}(\textbf{t})$.
We select the nine monomials \[\mydef{\mathcal P} = 
\{1,w,z,w^2,zw,z^2,w^3,zw^2,z^3\}\]
and let $\mydef{\mathbb V_{\mathcal P}}$ be the corresponding
$9\times9$ submatrix of $\mathbb V$. Similarly $\textbf{c}_{\mathcal P}(\textbf{t})$ is the vector selecting those terms from $\textbf{c}(\textbf{t})$.  Explicitly,
\[
\mathbb V_{\mathcal P}
=
\begin{pmatrix}
1&w_1&z_1&w_1^2&z_1w_1&z_1^2&w_1^3&z_1w_1^2&z_1^3\\
1&w_2&z_2&w_2^2&z_2w_2&z_2^2&w_2^3&z_2w_2^2&z_2^3\\
\vdots&\vdots&\vdots&\vdots&\vdots&\vdots&\vdots&\vdots&\vdots\\
1&w_9&z_9&w_9^2&z_9w_9&z_9^2&w_9^3&z_9w_9^2&z_9^3
\end{pmatrix}.
\]
Since
\[
\mathbb V\mathbf c
=
\mathbb V_{\mathcal P}\mathbf c_{\mathcal P}
+
\mathbb V_{\mathcal P^c}\mathbf c_{\mathcal P^c},
\]
multiplying the system by $\mathbb V_{\mathcal P}^{-1}$ gives the
equivalent system
\begin{align*}
\mydef{\mathcal F^{(2)}}
=
\mathbb V_{\mathcal P}^{-1}\mathcal F =\mathbb{V}_{\mathcal P}^{-1}\mathbb{V}_{\mathcal P} \textbf{c}_{\mathcal P}(\textbf{t}) + \mathbb{V}_{\mathcal P}^{-1}\mathbb{V}_{\mathcal P^c} \textbf{c}_{\mathcal P^c}(\textbf{t}) &=
\mathbf c_{\mathcal P}(\mathbf t)
+
\mathbb V_{\mathcal P}^{-1}
\mathbb V_{\mathcal P^c}
\mathbf c_{\mathcal P^c}(\mathbf t)\\
&= {\footnotesize{\begin{pmatrix}
c_{00}(\mathbf t)
+\displaystyle\sum_{(i,j)\notin\mathcal P}
\lambda_{1,ij}c_{ij}(\mathbf t)
\\[3pt]
c_{01}(\mathbf t)
+\displaystyle\sum_{(i,j)\notin\mathcal P}
\lambda_{2,ij}c_{ij}(\mathbf t)
\\[3pt]
c_{10}(\mathbf t)
+\displaystyle\sum_{(i,j)\notin\mathcal P}
\lambda_{3,ij}c_{ij}(\mathbf t)
\\
\vdots
\\
c_{30}(\mathbf t)
+\displaystyle\sum_{(i,j)\notin\mathcal P}
\lambda_{9,ij}c_{ij}(\mathbf t)
\end{pmatrix}}}
=
\mathbf 0.
\end{align*}
Here the $\lambda_{k,ij}$ are rational functions in the precision point coordinates. 

Thus the $j$th transformed equation contains its chosen pivot
coefficient with coefficient $1$ and none of the other eight pivot
coefficients. This is Gaussian elimination on nine columns of the
point-evaluation matrix.
The resulting equations have total degrees
$(
5,\ 4,\ 4,\ 4,\ 4,\ 4,\ 3,\ 3,\ 3,
)$ 
and the mixed volume of their Newton polytopes $\mydef{\Delta_{\bullet}^{(2)}} = (\Delta_1^{(2)},\ldots,\Delta_9^{(2)})$ is
\[
\textrm{MV}(\Delta_{\bullet}^{(2)}) = 6064.
\]
We tried each of the ${{16}\choose{9}}=11440$ pivot options $\mathcal P$. This gave the best reduction in mixed volume. 
\subsection{A birational change of variables}

The final reduction of the mixed volume to $5538$
comes from a birational change of coordinates on our affine chart
of $\textrm{Gr}(3,6)$.
Recall that the affine chart is represented by
\[
\mathbf M(\mathbf t)
=
\begin{pmatrix}
t_0&t_3&t_6&1&0&0\\
t_1&t_4&t_7&0&1&0\\
t_2&t_5&t_8&0&0&1
\end{pmatrix}.
\]
Simplifying notation, write  $\mydef{\mathbf c_0},\mydef{\mathbf c_1},\mydef{\mathbf c_2}$ for its first three
columns. On the open set $t_5\ne0$, introduce coordinates
$\mydef{\mathbf s}=(\mydef{s_1},\ldots,\mydef{s_9})$ by requiring
\[
\mathbf c_1
=
\frac{1}{s_1}
\begin{pmatrix}
s_3\\s_4\\1
\end{pmatrix},
\qquad
\mathbf c_0
=
s_2\mathbf c_1+
\begin{pmatrix}
s_8\\s_9\\0
\end{pmatrix},
\qquad
\mathbf c_2
=
s_1s_7\mathbf c_1-
\begin{pmatrix}
s_5\\s_6\\0
\end{pmatrix}.
\]
Note that the first assignment is to $\textbf{c}_1$ and not $\textbf{c}_0$. Equivalently,
{\footnotesize{\[
s_1=\frac{1}{t_5},\quad s_2=\frac{t_2}{t_5},\quad s_3=\frac{t_3}{t_5},\quad s_4=\frac{t_4}{t_5},\quad s_5=\frac{t_3t_8}{t_5}-t_6,\quad s_6=\frac{t_4t_8}{t_5}-t_7,\quad s_7=t_8,\quad s_8=t_0-\frac{t_2t_3}{t_5},\quad s_9=t_1-\frac{t_2t_4}{t_5}.
\]}}
\noindent The inverse change of coordinates is
\[
\begin{aligned}
t_0&=\frac{s_2s_3}{s_1}+s_8,
&
t_1&=\frac{s_2s_4}{s_1}+s_9,
&
t_2&=\frac{s_2}{s_1},
\\
t_3&=\frac{s_3}{s_1},
&
t_4&=\frac{s_4}{s_1},
&
t_5&=\frac{1}{s_1},
\\
t_6&=s_3s_7-s_5,
&
t_7&=s_4s_7-s_6,
&
t_8&=s_7.
\end{aligned}
\]
Thus the change is an isomorphism between the open subsets
$t_5\ne0$ and $s_1\ne0$.
Substituting $\mathbf t=\mathbf t(\mathbf s)$ into the coupler
polynomial introduces powers of $s_1$ in the denominator, which we clear:
\[
\mydef{G_{\mathbf s}(z,w)}
=
s_1^2F_{\mathbf t(\mathbf s)}(z,w)
=
\sum_{i=0}^3\sum_{j=0}^3
\mydef{g_{ij}(\mathbf s)}z^iw^j.
\]
Applying the same Vandermonde reduction as before, using the same pivots, we obtain the equations 
\begin{align*}
\mydef{\mathcal F^{(3)}(\textbf{s};\textbf{z},\textbf{w})}
&=
\textbf{g}_{\mathcal P}(\mathbf s)
+
\mathbb V_{\mathcal P}^{-1}
\mathbb V_{\mathcal P^c}
\textbf{g}_{\mathcal P^c}(\mathbf s)\\
&= {\footnotesize{\begin{pmatrix}
g_{00}(\mathbf s)
+\displaystyle\sum_{(i,j)\notin\mathcal P}
\lambda_{1,ij}g_{ij}(\mathbf s)
\\[3pt]
g_{01}(\mathbf s)
+\displaystyle\sum_{(i,j)\notin\mathcal P}
\lambda_{2,ij}g_{ij}(\mathbf s)
\\[3pt]
g_{10}(\mathbf s)
+\displaystyle\sum_{(i,j)\notin\mathcal P}
\lambda_{3,ij}g_{ij}(\mathbf s)
\\
\vdots
\\
g_{30}(\mathbf s)
+\displaystyle\sum_{(i,j)\notin\mathcal P}
\lambda_{9,ij}g_{ij}(\mathbf s)
\end{pmatrix}}}
=
\mathbf 0.
\end{align*} Here the $\mydef{\lambda_{k,ij}(\textbf{z},\textbf{w})}$ are rational functions of the precision point coordinates. Their Newton polytopes have mixed volume
\[
\textrm{MV}(\Delta_1^{(3)},\ldots,\Delta_9^{(3)}) = {5538}=1442+2^{12}.
\]
Note that there are $63$ newly introduced $\lambda_{k,{ij}}$ but only $18$ precision coordinates. We call the system with the $\boldsymbol \lambda$ generic parameters $\mydef{\mathcal F^{(4)}(\textbf{s}; \boldsymbol \lambda)} = \textbf{0}$. One may expect  ${\mathcal F^{(4)}(\textbf{s}; \boldsymbol \lambda)} = \textbf{0}$ to have a larger generic solution count than $\mathcal F^{(3)}(\textbf{s};\textbf{z},\textbf{w})$, as it is more general than the case where they are restricted. The following lemma justifies this bound, which we will see later is tight anyway. 
\begin{lemma}
\label{lem:lower_than_lambda}
The generic number of regular solutions to $\mathcal F^{(4)}(\textbf{s};\boldsymbol \lambda)=\textbf{0}$ is at least the generic number of regular solutions to the nine-point path synthesis problem. 
\end{lemma}
\begin{proof}
By construction, the system $\mathcal F^{(3)}(\textbf{s};\textbf{z},\textbf{w})$ is a specialization of $\mathcal  F^{(4)}(\textbf{s};\boldsymbol \lambda)$. Since both systems are square, by the parameter continuation theorem \cite{BorovikBreiding2025}, the former has a generic regular solution count of at most the latter. 
\end{proof}

\begin{remark}[Alternative pivots]
The set $
\mydef{\mathcal P'}
=
\{
c_{00},c_{01},c_{02},c_{03},
c_{10},c_{11},c_{12},c_{13},c_{20}
\},$ gives a pivot choice which
produces a system whose Newton polytopes have mixed volume  $ {5329}=73^2 = 1442+3887$.
We use $\mathcal P$ for the boundary analysis instead of $\mathcal P'$. 
\end{remark}We conclude this section by recording the coefficient polynomials
$g_{ij}(\mathbf s)$ using the pivots $\mathcal P$ whose  corresponding $g_{ij}$ are underlined:
{\scriptsize
\begin{align*}
\underline{g_{00}}={}&s_1s_7s_8^2+s_2s_5s_8+s_5s_6s_8-s_5^2s_9,\\
\underline{g_{01}}={}&-s_1s_5s_7s_8-s_2s_5^2+s_2s_3s_8+s_4s_5s_8+s_3s_6s_8+s_1s_8^2-2s_3s_5s_9,\\
\underline{g_{10}}={}&-s_1s_2s_7s_8+s_1s_6s_7s_8-2s_1s_5s_7s_9-s_2^2s_5-s_2s_5s_6+s_5s_8,\\
\underline{g_{02}}={}&s_1s_3s_7s_8-s_2s_3s_5-s_4s_5^2+s_3s_5s_6+s_3s_4s_8-2s_1s_5s_8-s_3^2s_9,\\
\underline{g_{11}}={}&-s_1^2s_7^2s_8-s_1s_2s_5s_7+s_1s_4s_7s_8-2s_1s_3s_7s_9-s_2^2s_3-s_2s_4s_5-s_2s_3s_6-s_1s_2s_8+s_1s_6s_8-2s_1s_5s_9-2s_5^2+s_3s_8,\\
\underline{g_{20}}={}&-s_1^2s_7^2s_9-s_1s_2s_6s_7-s_1s_7s_8-2s_2s_5-s_5s_6,\\
\underline{g_{03}}={}&-s_1s_3s_5s_7-s_3s_4s_5+s_1s_5^2+s_3^2s_6,\\
\underline{g_{12}}={}&-2s_1s_2s_3s_7-2s_1s_4s_5s_7+s_1s_3s_6s_7-s_1^2s_7s_8-s_2s_3s_4+s_1s_2s_5+s_1s_5s_6+s_1s_4s_8-2s_1s_3s_9-3s_3s_5,\\
g_{21}={}&-s_1s_2s_4s_7-2s_1^2s_7s_9-s_1s_2s_6-3s_1s_5s_7-2s_2s_3-s_4s_5-s_3s_6-s_1s_8,\\
\underline{g_{30}}={}&-s_1s_6s_7-s_5,\\
g_{13}={}&-s_1^2s_3s_7^2-s_1s_3s_4s_7+s_1^2s_5s_7-s_1s_4s_5+2s_1s_3s_6-s_3^2,\\
g_{22}={}&-s_1^2s_4s_7^2+s_1^2s_6s_7-s_1s_2s_4-4s_1s_3s_7-s_1^2s_9-s_3s_4-s_1s_5,\\
g_{31}={}&-s_1^2s_7^2-s_1s_4s_7-s_1s_6-s_3,\\
g_{23}={}&-s_1^2s_4s_7+s_1^2s_6-2s_1s_3,\\
g_{32}={}&-2s_1^2s_7-s_1s_4,\\
g_{33}={}&-s_1^2.
\end{align*}
}
\section{Tropical BooKKeeping}
\label{sec:TropicalBookKeeping}
We describe our method of \textit{Tropical BooKKeeping}\footnote{Capitalization is intentional} for obtaining an upper bound on the number sought by Alt's problem. Let
$
\mydef{\mathbb T}=(\mathbb C^\times)^n
$
be the \mydef{$n$-dimensional algebraic torus}, and consider a
square Laurent polynomial system
\[
\mathcal F=(f_1,\ldots,f_n)
\]
in $\mathbb{C}[x_1,\ldots,x_n]=\mydef{\mathbb{C}[\textbf{x}]}$ with Newton polytopes
$
(\mydef{\Delta_1},\ldots,\mydef{\Delta_n}) 
$
in $\mathbb{R}^n$. Write
\[
\mydef{\Delta}
=
\Delta_1+\cdots+\Delta_n
\]
for their Minkowski sum.

\subsection{The BKK machinery}
The Bernstein-Kouchnirenko-Khovanskii theorem (see \Cref{prop:BKK}) counts the number of solutions to a generic square sparse polynomial system, that is, a generic system with a prescribed support. The count is the mixed volume of the convex hulls of those supports, which are the Newton polytopes of the system. Special sparse systems can have fewer isolated solutions in the torus, counted with multiplicity, but never more. Geometrically, this occurs when the branches of the generic system leave the non-compact solution space $\mathbb{T}$. A compactification of $\mathbb{T}$, however, explains the loci where those solutions can go, and how those loci relate to the combinatorics of the Newton polytopes. Specifically, solutions which go toward the  boundary loci correspond to solutions to \textit{facial systems} of the original system. 

Fix a sparse polynomial system $\mathcal F=(f_1,\ldots,f_n)$ supported on $\mydef{\mathcal A_{\bullet}}=(\mathcal A_1,\ldots,\mathcal A_n)$ with Newton polytopes $\mydef{\Delta_{\bullet}} = (\Delta_1,\ldots,\Delta_n)$. Write 
\[
\mydef{f_i}=\sum_{a\in \mathcal  A_i}c_{i,a}\textbf{x}^a,
\]
and consider a weight vector $\mydef{\omega}\in\mathbb R^n$. The weight $\omega$  exposes a face of each Newton polytope:
\[
\mydef{\Delta_i^\omega}
=
\left\{
a\in\Delta_i
\mathrel{\big|}
\langle\omega,a\rangle
=
\min_{b\in\Delta_i}\langle\omega,b\rangle
\right\}
\]
and therefore identifies the part of the polynomial $f_i$ involving monomials on that face:
\[
\mydef{f_i^\omega}
=
\sum_{\substack{a\in \mathcal A_i \cap \Delta_i^\omega}}
c_{i,a}\mathbf{x}^a.
\]
We call $f_i^\omega$ the \mydef{initial form} of $f_i$ in the direction $\omega$. 
The \mydef{facial system} in the direction of $\omega$ is
\[
\mydef{\mathcal F^\omega}
=
\left(
f_1^\omega,\ldots,f_n^\omega
\right).
\]
Positive multiples of $\omega$ determine the same faces
$\Delta_i^\omega$ and hence the same facial system.

\begin{proposition}[Full BKK Theorem \cite{Bernstein1975}]
\label{prop:BKK}
Let $\mathcal F=(f_1,\ldots,f_n)$ be a square Laurent polynomial system with Newton polytopes
$\Delta_{\bullet}=(\Delta_1,\ldots,\Delta_n)$ of positive mixed volume. The number of isolated solutions in $\mathbb{T}$, counted with multiplicity, is 
\[
\textrm{MV}(\Delta_1,\ldots,\Delta_n)
\]
if and only if there is no nonzero vector $\omega \in \mathbb{R}^n$ such that the facial system $\mathcal F^{\omega}$ has a solution in $(\mathbb{C}^\times)^n$. For a generic system with Newton polytopes $\Delta_{\bullet}$, all torus solutions are simple. 
\end{proposition}
The part of \Cref{prop:BKK} which implies that a generic system has mixed volume many simple torus solutions is the \mydef{BKK Theorem} \cite{Bernstein1975,Kouchnirenko1976}. We call such a system \mydef{BKK-generic}.  The part which blames the facial systems for loss is sometimes referred to as \mydef{Bernstein's other theorem} or \mydef{Bernstein's second theorem}.

\subsection{The sparse homotopy}
Suppose $\mydef{\mathcal G}=(g_1,\ldots,g_n)$ is a generic sparse system with support $\mydef{\mathcal A_{\bullet}}=(\mathcal A_1,\ldots,\mathcal A_n)$, and consider $\mathcal F=(f_1,\ldots,f_n)$, a system with Newton polytopes  $\mydef{\Delta_{\bullet}}=(\Delta_1,\ldots,\Delta_n) = (\textrm{conv}(\mathcal A_1),\ldots,\textrm{conv}(\mathcal A_n))$ whose solution count we would like to bound from above. Importantly, $\mathcal F$ may fail to be BKK-generic: it may be missing terms that appear in $\mathcal G$. The Newton polytopes of $\mathcal F$, however,  must coincide  with those of $\mathcal G$. In fact, we assume that the only missing monomials in $\mathcal F$, if any, are interior ones.  The idea we explore here is a degeneration from $\mathcal G$ to $\mathcal F$ represented by the straight-line homotopy as $\varepsilon \to 0$:
\[
\mydef{\mathcal H_{\varepsilon}} = (1-\varepsilon) \mathcal F+ \varepsilon \mathcal G.
\]
We call this a \mydef{sparse homotopy}. For generic $\varepsilon \neq 0$, the system $\mathcal H_{\varepsilon}$ has $d = \textrm{MV}(\Delta_{\bullet})$ distinct solutions in the torus, as it is BKK-generic. We write a solution branch of $\mathcal H_{\varepsilon}=\textbf{0}$ as 
\[
\mydef{\textbf{s}(\varepsilon)} = (s_1(\varepsilon),\ldots,s_n(\varepsilon)) \in \mathbb{T}.
\]
Such a solution branch admits a Puiseux series expansion as 
\[
\mydef{\textbf{s}(\varepsilon)} = \left(\varepsilon^{\omega_1}(\mydef{\widehat s_1}+ \text{ higher-order terms in }\varepsilon), \ldots,  \varepsilon^{\omega_n}(\mydef{\widehat s_n}+ \text{ higher-order terms in }\varepsilon)\right)
\]
for $\omega \in \mathbb{Q}^n$ (e.g., see \cite[Section~1.4, Theorem~1.7]{Sturmfels2002}). We call $\omega$ the \mydef{tropical vector} of the branch $\textbf{s}(\varepsilon)$. For $\varepsilon \neq 0$, one may change variables $(s_1,\ldots,s_n) \to (\varepsilon^{-\omega_1}s_1,\ldots,\varepsilon^{-\omega_n}s_n)$,  making  $\widehat{\textbf s} = (\widehat s_1,\ldots,\widehat s_n)$ the limit as $\varepsilon \to 0$. Note that, by construction, $\widehat{\textbf{s}}$ belongs to the torus.  Substituting these expansions into $\mathcal H_\varepsilon$ gives
\[
0
=
H_{i,\varepsilon}(\textbf{s}(\varepsilon))
=
\varepsilon^{\min_{b\in\operatorname{supp}(f_i)}\langle\omega,b\rangle}
\left(
f_i^\omega(\widehat{\textbf{s}})
+\text{ higher-order terms in }\varepsilon
\right).
\]
Thus, $
f_i^\omega(\widehat{\textbf{s}})=0,$
and therefore $\mathcal F^\omega(\widehat{\textbf{s}})=\textbf{0}$. This verifies \Cref{prop:BKK}: each of the $\textrm{MV}(\Delta_{\bullet})$ branches has some tropical vector which is either $\omega = \textbf{0}$ implying that the solution persists in the torus in the limit or $\omega \neq \textbf{0}$ implying the solution leaves the torus in the limit and produces a solution of a facial system. Write $\mydef{\mathcal B_{\omega}}$ for the set of branches with tropical vector $\omega$. The following lemma follows from this discussion.  
\begin{lemma}
\label{lem:bookkeeping}
Fix $\Delta_{\bullet}$ and let $\mathcal G$ be a generic sparse system with support $\mathcal A_{\bullet}$ and Newton polytopes $\Delta_{\bullet}$. Let $\mathcal F$ be a system whose support agrees with $\mathcal G$, except possibly at interior monomials of the Newton polytopes. Let $\mathcal H_{\varepsilon}=(1-\varepsilon)\mathcal F + \varepsilon \mathcal G$ be a sparse homotopy. Let $d$ be the number of isolated solutions to $\mathcal F = \textbf{0}$ in the torus, counted with multiplicity. Then 
\[
\textrm{MV}(\Delta_{\bullet}) = \sum_{\omega} |{\mathcal B_{\omega}}| = |{\mathcal B_{\textbf{0}}}| + \sum_{\omega \neq \textbf{0}} |{\mathcal B_{\omega}}|  \geq  d +  \sum_{\omega \neq \textbf{0}} |{\mathcal B_{\omega}}| 
\]
The $\omega$ in the last summand may be taken to range over directions exposing faces of the Minkowski sum $\Delta$. Any  lower bound $\delta$ for that sum produces an upper bound for $d$:
\[
d \leq \textrm{MV}(\Delta_{\bullet}) - \delta.
\]
\end{lemma}
\begin{proof}
The number of branches is $\textrm{MV}(\Delta_{\bullet})$ by BKK \Cref{prop:BKK}. The tropical machinery partitions them by tropical vector, giving the first equality. The next equality is just reindexing $\mathbb{R}^n \to \{\textbf{0}\} \cup \mathbb{R}^n-\{\textbf{0}\}$. Finally, an isolated torus solution to $\mathcal F=\textbf{0}$ must have tropical vector $\textbf{0}$. The converse may fail if $\mathcal F$ has a positive-dimensional component, giving only the bound stated in the result rather than equality. That $\omega$ may be taken over the facial directions of $\Delta = \Delta_1+\cdots+\Delta_n$ is immediate, since otherwise $\omega$ would expose a mere monomial in each polynomial and produce an empty torus solution set. Finally, 
\[\delta \leq \sum_{\omega \neq \textbf{0}} |{\mathcal B_{\omega}}| \implies \textrm{MV}(\Delta_{\bullet}) \geq d + \delta \implies d \leq \textrm{MV}(\Delta_{\bullet})-\delta.
\]
\end{proof}
\noindent We call the number $\sum_{\omega \neq \textbf{0}} |\mathcal B_{\omega}|$ the \mydef{BKK defect} of the system.

\subsection{Proving lower bounds on the BKK defect}
Let $\omega\in\mathbb Q^n$ be a candidate tropical vector. Choose a
positive integer $e$ for which $\mydef{\nu}=e\omega\in\mathbb Z^n$.
We usually take the smallest such $e$, but allow a multiple when
needed for the power series expansions below. Make the change of
variables from the original coordinates $(s_1,\ldots,s_n)$ to rescaled coordinates $\mydef{\textbf{y}} = (y_1,\ldots,y_n)$:
\[
\varepsilon=\delta^e,\qquad
s_j=\varepsilon^{\omega_j} y_j = \delta^{\nu_j}y_j,\quad j=1,\ldots,n.
\]
This way, even if some $s_i$ go to $\infty$ or $0$ as $\varepsilon \to 0$, the $y_i$ may approach finite nonzero numbers. That is the intention, and we will see this is the case if the tropical vector of $\textbf{s}$ is $\omega$.  To extract the resulting potential torus point, for each equation, let
\[
\mydef{a_{i,\omega}}
=
\operatorname{ord}_{\delta}
H_{i,\delta^e}
\left(
\delta^{\nu_1}y_1,\ldots,\delta^{\nu_n}y_n
\right)
\]
be the smallest power of $\delta$ appearing after this substitution.
The \mydef{rescaled homotopy at $\omega$} is
\[
\mydef{\widetilde H_{\omega,i}(\mathbf y,\delta)}
=
\delta^{-a_{i,\omega}}
H_{i,\delta^e}
\left(
\delta^{\nu_1}y_1,\ldots,\delta^{\nu_n}y_n
\right).
\] This evaluation encodes more than the evaluation $\mathcal F(\textbf{s}(\varepsilon))$ from earlier as it retains  differential information about  the homotopy itself. 
By the definition of $a_{i,\omega}$, this expression is regular at
$\delta=0$ and does not vanish identically there.
Define the \mydef{first compatibility equations}
\[
\mydef{K_{\omega,i}(\mathbf y)}
=
\widetilde H_{\omega,i}(\mathbf y,0),
\qquad i=1,\ldots,n,
\]
and write $\mydef{\mathcal K_\omega}=(K_{\omega,1},\ldots,K_{\omega,n})$
for the \mydef{first compatibility system}. In the present setting, this is
precisely the facial system $\mathcal K_\omega=\mathcal F^\omega$.
If a branch has tropical vector $\omega$, then its rescaled coordinates
have an expansion
\[
y_j(\delta)=\mydef{\widehat y_j}+\text{higher-order terms in }\delta,
\qquad \widehat y_j\ne0,
\]
and its \mydef{leading coefficient} $\widehat{\mathbf y}\in(\mathbb C^\times)^n$
necessarily satisfies $\mathcal K_\omega(\widehat{\mathbf y})=\mathbf0$.

When $\omega\ne\mathbf0$, weighted homogeneity preserves its solution
set under the torus action of $(y_1,\ldots,y_n)\mapsto(t^{\nu_1}y_1,\ldots,t^{\nu_n}y_n)$
for $t\in\mathbb C^\times$, so every torus solution lies in a
positive-dimensional family. The first compatibility system therefore
has no isolated nonsingular torus solutions.
Its role is instead to locate the possible limiting configurations.
For each boundary regime, 
\begin{itemize}
\item we identify a part of
$\mathcal V(\mathcal K_\omega)$, 
\item we introduce coordinates $\boldsymbol\xi$ along this family, for example by a parametrization
\[
\boldsymbol\xi\longmapsto
\mathbf y(\boldsymbol\xi)
=
\bigl(y_1(\boldsymbol\xi),\ldots,y_n(\boldsymbol\xi)\bigr);
\]
\item we introduce coordinates $\mathbf u$ transverse to the family by allowing some of the rescaled coordinates to vary as
\[
y_i
=
y_i(\boldsymbol\xi)
+
\delta^{m_i}u_i,
\qquad m_i>0,
\]
and, when necessary, retain an additional coordinate for the weighted scale, equivalently choosing a point in an infinite torus orbit
\item we substitute these local coordinates into the full rescaled homotopy
$\widetilde{\mathcal H}_{\omega}$;
\item we continue the expansion in $\delta$, equating successive coefficients to zero, until we obtain a finite square system whose simple roots can be certified.
\end{itemize}
The resulting equations are called \mydef{higher-order compatibility equations}. 
Similar ideas appear in
\cite{JensenMarkwigMarkwig2008,HuberSturmfels1995}.
We package these choices into a rational substitution
\[
\mathbf s=\mathbf S(\boldsymbol\xi,\mathbf u,\delta),
\]
where $(\boldsymbol\xi,\mathbf u)$ collects the facial coordinates, transverse
coefficients, scale coordinate, and any other coefficients retained in
the expansion. The limiting value at $\mathbf u(0)$ records the finite
Puiseux data being solved for, while allowing $\mathbf u$ itself to vary
with $\delta$ accounts for still higher-order terms.
We require that, for each fixed sufficiently small $\delta\ne0$, the
map $
(\boldsymbol\xi,\mathbf u)
\longmapsto
\mathbf S(\boldsymbol\xi,\mathbf u,\delta)
$ 
is injective on the domain under consideration. We also record a square
matrix $M(\boldsymbol\xi,\mathbf u,\delta)$ of rational functions such that
\[
\mathcal L(\boldsymbol\xi,\mathbf u,\delta)
=
M(\boldsymbol\xi,\mathbf u,\delta)
\mathcal H_{\delta^e}
\bigl(\mathbf S(\boldsymbol\xi,\mathbf u,\delta)\bigr),
\]
where $M$ is invertible for $\delta\ne0$ on this domain and
$\mathcal L$ is regular at $\delta=0$. The matrix $M$ records the row
operations and divisions by powers of $\delta$ used in the calculation.
We call a square polynomial system of the form
$
\mathcal K(\boldsymbol\xi,\mathbf u)
=
\mathcal L(\boldsymbol\xi,\mathbf u,0)
=
\mathbf0$
a \mydef{finite compatibility system} for this substitution.
For a proposed tropical vector $\omega=\nu/e$, write
\[
S_j(\boldsymbol\xi,\mathbf u,\delta)
=
\delta^{\nu_j}
\bigl(c_j(\boldsymbol\xi,\mathbf u)+\delta R_j(\boldsymbol\xi,\mathbf u,\delta)\bigr),
\qquad j=1,\ldots,n.
\]
Suppose $(\widehat{\boldsymbol\xi},\widehat{\mathbf{u}})$ is a
 solution of a compatibility system $\mathcal K$. Let $q(\boldsymbol\xi,\mathbf{u})$ be a polynomial such  that
\[
q(\widehat{\boldsymbol\xi},\widehat{\mathbf u})\ne0
\]
ensures that $\mathcal L$ and all $R_j$ are regular near
$(\widehat{\boldsymbol\xi},\widehat{\mathbf u},0)$, that all $c_j$ are
regular and nonzero there, and that $M$ is defined and invertible nearby
for $\delta\ne0$. Such a polynomial may be obtained by multiplying together the factors
whose nonvanishing is required in the construction: denominators of the
substitution and row operations, pivots used to solve for auxiliary
variables, and leading coefficients that must remain nonzero to preserve
the prescribed tropical vector. Any powers of $\delta$ are discarded,
since the guard is a condition on the limiting variables
$(\boldsymbol\xi,\mathbf u)$. We call $q$ a \mydef{guard polynomial}. The displayed
identity, injectivity of the substitution, and these properties of
$q$ are verified by exact algebra.
Variables may be eliminated before certification. In that case, we
require that the eliminated equations determine those variables
uniquely as rational functions of the remaining variables on the
guarded open set, and that their Jacobian with respect to the
eliminated variables is invertible there. Substituting these
expressions in the remaining equations gives a reduced system.
Each simple solution of this reduced system whose recovered
coordinates avoid $q=0$ then determines a unique simple solution of
$\mathcal K=\mathbf0$. The guard includes the denominator and pivot
factors needed for this elimination.
\begin{lemma}[Lifting certified compatibility solutions]
\label{lem:lift_compatibility}
With the preceding notation and hypotheses, suppose interval
certification proves, either directly or through a reduced system
as above, that $\mathcal K=\mathbf0$ has $d$ distinct solutions at
which $q$ does not vanish. Then the sparse homotopy has at least $d$
distinct Puiseux solutions with tropical vector $\omega$. These
solutions give $d$ distinct points for each sufficiently small fixed
$\delta\ne0$, with $\varepsilon=\delta^e$.
\end{lemma}
\begin{proof}
Let $(\widehat{\boldsymbol\xi},\widehat{\mathbf u})$ be one of these
solutions. Certification proves that it is simple, using the preceding
elimination argument when a reduced system is certified. The implicit
function theorem therefore gives unique convergent power series
$\boldsymbol\xi(\delta)$ and $\mathbf u(\delta)$ with $
\boldsymbol\xi(0)=\widehat{\boldsymbol\xi},
\mathbf u(0)=\widehat{\mathbf u}$, 
and $
\mathcal L(\boldsymbol\xi(\delta),\mathbf u(\delta),\delta)=\mathbf0.$
For sufficiently small $\delta\ne0$, the identity defining
$\mathcal L$ and invertibility of $M$ imply
\[
\mathcal H_{\delta^e}
\bigl(
\mathbf S(\boldsymbol\xi(\delta),\mathbf u(\delta),\delta)
\bigr)
=
\mathbf0.
\]
Moreover,
\[
S_j(\boldsymbol\xi(\delta),\mathbf u(\delta),\delta)
=
\delta^{\nu_j}
\bigl(
c_j(\widehat{\boldsymbol\xi},\widehat{\mathbf u})
+
O(\delta)
\bigr).
\]
The guard ensures that these leading coefficients are nonzero, so the
tropical vector is $\nu/e=\omega$.
The $d$ solutions of $\mathcal K=\mathbf0$ have distinct limits and
hence give distinct pairs
$(\boldsymbol\xi(\delta),\mathbf u(\delta))$ for sufficiently small
$\delta$. Injectivity of the substitution then gives $d$ distinct
solutions of the original homotopy.
\end{proof}
We are able to apply \Cref{lem:lift_compatibility} to account for
$4096$ solutions. The setup is this.
Construct the homotopy
\[
\mathcal H_\varepsilon
=
(1-\varepsilon)\mathcal F^{(4)}
(\mathbf s;\boldsymbol\lambda)
+
\varepsilon\mathcal G(\mathbf s).
\]
For each proposed boundary regime, we are given a positive integer $e$
and a rational substitution
$\mathbf s=\mathbf S(\mathbf u,\delta)$. We then make the
regime-specific ramification substitution $\varepsilon=\delta^e$.

The pair $(e,\mathbf S)$ encodes the proposed asymptotic direction.
Indeed, suppose
\[
S_j(\mathbf u,\delta)
=
\delta^{\nu_j}
\bigl(c_j(\mathbf u)+O(\delta)\bigr).
\]
On the locus where every $c_j(\mathbf u)$ is nonzero, the corresponding
tropical vector is $\omega=\nu/e$. The integer $e$ need not be the
smallest denominator-clearing ramification index, since a larger value
may be needed to express the higher-order corrections using integral
powers of $\delta$.

We now substitute $\varepsilon=\delta^e$ and
$\mathbf s=\mathbf S(\mathbf u,\delta)$ into the homotopy. At this point, division by the lowest powers of $\delta$ in each equation is not sufficient to make the solutions regular. Indeed, some polynomials with the same valuation may cancel in the system. Exact row
operations and divisions produce a transformed system
\[
\mathcal L(\mathbf u,\delta)
=
M(\mathbf u,\delta)
\mathcal H_{\delta^e}
\bigl(\mathbf S(\mathbf u,\delta)\bigr)
\]
that is regular at $\delta=0$.

Throughout the construction, every denominator, pivot, divided factor,
reconstruction denominator, and leading coefficient required to be
nonzero is accumulated into a guard polynomial $q(\mathbf u)$. Thus the
guard is built by the verification procedure rather than imposed in
advance.

The finite compatibility system is
\[
\mathcal K(\mathbf u;\boldsymbol\lambda)
=
\mathcal L(\mathbf u,0;\boldsymbol\lambda)
=
\mathbf0.
\]
Here we have consolidated all facial and transverse coordinates into
the single tuple $\mathbf u$. The result is a square polynomial system
over $\mathbb Q[\boldsymbol\lambda]$. Our job is to  show a lower bound on the number of simple roots, generically.

Now, \Cref{lem:lift_compatibility} suggests that we need only one
certified instance of
$\mathcal K(\mathbf u;\boldsymbol\lambda)$. Our auxiliary code first
retains $\boldsymbol\lambda$ as parameters and solves a fiber over a
\texttt{Float64} parameter value
$\boldsymbol\lambda^{\mathrm{fl}}$ using monodromy. It performs a
preliminary numerical check that the resulting solutions are distinct
and nonsingular.

The code then chooses a nearby exact rational parameter value
$\boldsymbol\lambda^*\in\mathbb Q^{63}$ and tracks the solutions from
the fiber over $\boldsymbol\lambda^{\mathrm{fl}}$ to the fiber over
$\boldsymbol\lambda^*$. It substitutes $\boldsymbol\lambda^*$ into
$\mathcal K$ symbolically and certifies the tracked solutions using
interval arithmetic.

For every certified solution interval $\square$, the code also verifies
$0\notin q(\square)$. Thus all coordinate substitutions, divisions, pivots,
and reconstructions used to produce $\mathcal L$ are valid near the
certified solution, and all leading coefficients appearing in
$\mathbf S$ are nonzero. \Cref{lem:lift_compatibility} therefore
produces a solution branch of the original homotopy. The coordinatewise
orders encoded by $\mathbf S$, divided by $e$, give the tropical vector
of that branch.

This structure makes  the calculation a verification procedure rather than a
method whose correctness depends on how the boundary data were
discovered. An incorrect proposed direction or substitution will fail
the exact transformation check, fail to produce a system regular at
$\delta=0$, fail to have the required simple solutions and therefore fail certification, or solutions will meet the
guard locus. Conversely, if all of these checks succeed, the proposed
data give a valid collection of lifted branches, regardless of whether
they were found by geometric analysis, computer experimentation, or
simply guessed.

Thus the proposed data for each regime are
 $e$, $\mathcal L(\mathbf u,\delta)$, and
$\mathbf S(\mathbf u,\delta)$. The finite compatibility system
$\mathcal K$ and the guard polynomial $q$ are then obtained and verified
by the auxiliary code.

\section{Boundary details}
\label{sec:BoundaryDetails} 
In this section, we summarize the results we certified on $26$ boundary regimes.  We point the reader to the auxiliary code repository 
 \href{https://github.com/tbrysiewicz/AltsProblem}{here} for the code and certificates for the solutions to the compatibility systems. 
The numerical computations and certification were performed in \texttt{HomotopyContinuation.jl} \cite{HCjl}. We remark that the newly included option \texttt{duplicate\_check := certified} in their \texttt{monodromy\_solve} function was especially helpful in obtaining these certifications. Indeed, many of these solution sets are numerically ill-conditioned, and so without certification, the software may classify regular solutions as singular and not certify them. 

\begin{table}[!htpb]
{\scriptsize
\[\renewcommand{\arraystretch}{1.15}
\begin{array}{c|l|c|c|c}
\text{Boundary} & \text{Tropical vector} & \text{Regime} & \text{Count} & \text{Total contribution}\\
\hline
1 & \omega_1=(1,0,0,0,0,0,-1,0,1)/2 & \mathcal B_{1,1} & 1322 & 1322\\
\hline
2 & \omega_2=(1,0,1,1,1,0,0,0,0) & \mathcal B_{2,1} & 335 & 335\\
\hline
\multirow{2}{*}{3} & \multirow{2}{*}{\ensuremath{\omega_3=(-1,-1,-1,-1,-1,-1,0,-1,-1)}} & \mathcal B_{3,1} & 1648 & \multirow{2}{*}{2190}\\
\cline{3-4}
& & \mathcal B_{3,2} & 542 &\\
\hline
\multirow{4}{*}{4} & \omega_4=(-1,-1,-2,-1,-2,-1,0,-2,-1) & \mathcal B_{4,1} & 6 & \multirow{4}{*}{104}\\
\cline{2-4}
& \omega_4/2 & \mathcal B_{4,2} & 14 &\\
\cline{2-4}
& \omega_4/3 & \mathcal B_{4,3} & 54 &\\
\cline{2-4}
& 4\omega_4/5 & \mathcal B_{4,4} & 30 &\\
\hline
\multirow{3}{*}{5} & \omega_5/3=(-5,-3,-8,-4,-7,-3,1,-6,-2)/3 & \mathcal B_{5,1} & 18 & \multirow{3}{*}{50}\\
\cline{2-4}
& \omega_5/5 & \mathcal B_{5,2} & 20 &\\
\cline{2-4}
& \omega_5/6 & \mathcal B_{5,3} & 12 &\\
\hline
6 & \omega_6/4=(-6,-4,-10,-5,-9,-4,1,-8,-3)/4 & \mathcal B_{6,1} & 16 & 16\\
\hline
7 & \omega_7/4=(-4,-2,-6,-3,-5,-2,1,-4,-1)/4 & \mathcal B_{7,1} & 16 & 16\\
\hline
8 & \omega_8/2=(-3,-1,-5,-3,-4,-2,1,-3,-1)/2 & \mathcal B_{8,1} & 4 & 4\\
\hline
9 & \omega_9=(1,-1,2,1,2,1,0,0,-1) & \mathcal B_{9,1} & 8 & 8\\
\hline
10 & \omega_{10}=(5,1,6,3,4,1,-2,2,2)/4 & \mathcal B_{10,1} & 8 & 8\\
\hline
\multirow{2}{*}{11} & \omega_{11,1}=(14,2,16,8,10,2,-6,4,-1)/6 & \mathcal B_{11,1} & 6 & \multirow{2}{*}{16}\\
\cline{2-4}
& \omega_{11,2}=(14,2,16,8,10,2,-6,4,3)/10 & \mathcal B_{11,2} & 10 &\\
\hline
\multirow{2}{*}{12} & \omega_{12,1}=(6,1,6,3,4,1,-3,2,0)/3 & \mathcal B_{12,1} & 6 & \multirow{2}{*}{10}\\
\cline{2-4}
& \omega_{12,2}=(9,1,10,5,6,1,-4,2,-1)/4 & \mathcal B_{12,2} & 4 &\\
\hline
13 & \omega_{13}=(-3,-2,-6,-3,-5,-2,0,-4,-1)/3 & \mathcal B_{13,1} & 6 & 6\\
\hline
14 & \omega_{14}=(-4,-3,-8,-4,-7,-3,0,-6,-2)/4 & \mathcal B_{14,1} & 4 & 4\\
\hline
15 & \omega_{15}=(-8,-5,-10,-6,-6,-2,4,-7,-3)/2 & \mathcal B_{15,1} & 2 & 2\\
\hline
16 & \omega_{16}=(-2,-1,-3,-2,-2,-1,1,-2,-1) & \mathcal B_{16,1} & 2 & 2\\
\hline
17 & \omega_{17}=(-4,-2,-5,-4,-2,-1,3,-3,-2) & \mathcal B_{17,1} & 1 & 1\\
\hline
18 & \omega_{18}=(-3,-1,-6,-4,-5,-3,1,-4,-2)/2 & \mathcal B_{18,1} & 2 & 2\\
\hline
& & & 4096 & 4096\\
\end{array}\]
}
\caption{The partition of $4096$ into the $26$ boundary regimes. }
\label{tab:boundary_census}
\end{table}For each regime in \Cref{tab:boundary_census}, the auxiliary code is given
a ramification index $e$ and a rational substitution
$\mathbf s=\mathbf S(\mathbf u,\delta)$. It substitutes into
$\mathcal H_{\delta^e}$, performs the prescribed row cancellations,
and divides out the resulting powers of $\delta$ to construct a square
system $\mathcal L(\mathbf u,\delta;\boldsymbol\lambda)$ that is regular
at $\delta=0$. The finite compatibility system is
$\mathcal K(\mathbf u;\boldsymbol\lambda)
=\mathcal L(\mathbf u,0;\boldsymbol\lambda)$.
Exact algebra verifies the cancellations and the identities relating
these systems to the original homotopy. A guard polynomial $q(\mathbf u)$
records the required nonvanishing denominators, elimination pivots,
and leading coefficients of the original coordinates.
The parametrized compatibility system is solved numerically using
monodromy. The solutions are then tracked to a nearby exact rational
parameter value, where interval arithmetic certifies distinct simple
solutions. On every certified solution interval $\square$, we also
verify $0\notin q(\square)$. Together with injectivity of the substitution
on the guarded punctured chart, \Cref{lem:lift_compatibility} therefore
gives distinct branches of the original homotopy with the prescribed
tropical vector.

One can install the package \texttt{AltCertify.jl} from our repository. Then loading the datum of a boundary regime is simple, and certifying the result is a single command. 

\begin{lstlisting}
julia> B = B11
Boundary regime: B_1, regime 1
  omega          = Rational{Int64}[1//2, 0, 0, 0, 0, 0, -1//2, 0, 1//2]
  e              = 2
  facial         = u[1, 2, 5, 6]
  transverse/etc = u[3, 4, 7, 8, 9]
  guard          = u[1]*u[2]*u[5]*u[6]

  substitution:
    s[1] = u[1]*delta[1]
    s[2] = u[4]*delta[1] - 2*u[6]
    s[3] = -u[1]*u[6]*delta[1] - u[3]*u[5]*delta[1] + u[5]^2 + u[8]*delta[1]^2
    s[4] = u[3]*delta[1] - 2*u[5]
    s[5] = -u[5]*u[6] + u[7]*delta[1]^2
    s[6] = u[6]
    s[7] = u[5]//(u[1]*delta[1])
    s[8] = -u[2]*u[5]*delta[1] + u[4]*u[6]*delta[1] - u[6]^2 + u[9]*delta[1]^2
    s[9] = u[2]*delta[1]

  extra row factors:
    none
\end{lstlisting} 

\begin{lstlisting}

julia> certify(B11)
1. MONODROMY SOLVE
monodromy solutions: 1322
2. SOFT CERTIFICATION AT FLOAT64 PARAMETERS
soft-certified distinct solutions: 1322
3. NEARBY EXACT RATIONAL PARAMETERS
maximum parameter displacement: 1.4060347687259386e-12
exact parameter vector constructed.
4. PARAMETER HOMOTOPY TO RATIONAL FIBER
successfully tracked solutions: 1322
5. HARD CERTIFICATION AT EXACT RATIONAL PARAMETERS
hard-certified distinct solutions: 1322
6. GUARD CHECK ON CERTIFIED BOXES
guard excludes zero on 1322 / 1322 boxes
PASS: 1322 distinct exact-rational certified solutions, all guarded.
\end{lstlisting}
Sometimes it takes several tries, but once it succeeds it writes all the certification data to a file.

The takeaway from these computations is \Cref{prop:boundary_total},
which states that at least $4096$ of the $5538$ branches of a generic
sparse perturbation approach the boundary as $\varepsilon\to0$, rather
than a torus solution of $\mathcal F^{(4)}$. Consequently, \Cref{cor:lambda_upper_bound} translates this to an upper bound for the number of isolated solutions to a generic instance of $\mathcal F^{(4)}=\textbf{0}$, namely $1442$. This is a consequence of \Cref{lem:bookkeeping}. We summarize our certified lower bound for the number of regular isolated solutions to a generic instance of the nine-point path synthesis problem in \Cref{lem:lower_bound}. In particular, we note that we use the same formulation of this problem as \cite{MSW}. It is significantly better conditioned than the equivalent system developed in this text. It solves for the $6\cdot 1442$ mechanisms instead of the $1442$ coupler curves. These results combine to prove \autoref{thm:alt}. 

\begin{proposition}
\label{prop:boundary_total}
For a generic member of the independent-parameter family $\mathcal F^{(4)}$, at least $4096$ of the $5538$ torus branches of a generic sparse perturbation approach the boundary as $\varepsilon\to0$.
\end{proposition}
\begin{proof}
The certified contributions established above are
\[1322,\ 335,\ 2190,\ 104,\ 50,\ 16,\ 16,\ 4,\ 8,\ 8,\ 16,\ 10,\ 6,\ 4,\ 2,\ 2,\ 1,\ 2.\]
Their sum is
\[1322+335+2190+104+50+16+16+4+8+8+16+10+6+4+2+2+1+2=4096.\]
Branches with different tropical vectors are disjoint. Boundary $3$ is the only repeated tropical vector in the list, and its ordinary and special regimes are disjoint by the conditions which define them $\xi_2-\xi_4\ne0$ and $\xi_2=\xi_4$, respectively under some parametrization of the facial zero-set. For each regime the lower bound holds on a nonempty Zariski-open subset of the same target-and-perturbation parameter space. Intersecting these finitely many open sets makes the contributions simultaneous. For a fixed sufficiently small nonzero $\varepsilon$, the $4096$ reconstructed points are therefore distinct and may all be subtracted from the mixed-volume count.
\end{proof}
\begin{corollary}
\label{cor:lambda_upper_bound}
A generic system $\mathcal F^{(4)}(\mathbf s,\boldsymbol\lambda)=\mathbf0$ has at most
\[5538-4096=1442\]
isolated solutions in $(\mathbb C^\times)^9$, counted with multiplicity. Consequently, the generic nine-point system $\mathcal F^{(3)}=\mathbf0$ also has at most $1442$ isolated torus solutions.
\end{corollary}
\begin{proof}
The first assertion follows from \Cref{lem:bookkeeping} and \Cref{prop:boundary_total}. The system $\mathcal F^{(3)}$ is obtained from $\mathcal F^{(4)}$ by the specialization $\boldsymbol\lambda=\mathbb V_{\mathcal P}^{-1}\mathbb V_{\mathcal P^c}$. The bound on the nine-point path synthesis degree follows from the parameter continuation theorem \cite{BorovikBreiding2025}.
\end{proof}

\begin{lemma}
\label{lem:lower_bound}
The nine-point path synthesis problem generically admits at least $1442$ coupler curves.
\end{lemma}
\begin{proof}
We solve the system in \cite{MSW} and obtain $6 \cdot 1442$ certified mechanisms drawing coupler curves through nine points in the plane. Thus, the general count of that system must be at least this amount. Accounting for Roberts cognates proves the result. 
\end{proof}
We have now bounded the count in both directions by $1442$ and have thus proved \Cref{thm:alt}.
\begin{proof}[Proof of \Cref{thm:alt}]
Let $d_{\textrm{Alt}}$ be the generic number of four-bar coupler curves through nine generic precision points in the plane. By the construction in \Cref{sec:SparseReduction}, these curves correspond to isolated torus solutions of $\mathcal F^{(3)}=\mathbf0$ on the open set where the correspondence is defined. Therefore \Cref{cor:lambda_upper_bound} gives $d_{\textrm{Alt}}\leq1442$.

On the other hand, \texttt{certify\_alt\_mechanisms.jl} certifies $6\cdot 1442$ distinct regular solutions to one instance of the formulation of the nine-point path synthesis problem in \cite{MSW} in the mechanism coordinates. By \cite{BorovikBreiding2025}, this is a lower bound for the generic number of mechanisms, and so $1442$ is a lower bound for the generic number of coupler curves.  The upper and lower bounds coincide, proving that $d_{\textrm{Alt}}=1442$.
\end{proof}

\newpage

\bibliographystyle{plainnat}
\bibliography{references}

\newpage

\section{Appendix}
\label{sec:Appendix}

\begin{lemma}
\label{lem:tau_in_elliptic}
Let $E$ be a smooth projective  curve of genus one over $\mathbb C$,
and let $\tau$ be a fixed-point-free involution of $E$.
Choose an origin $O\in E$, and denote the resulting group law by
$\oplus$. Then there is a nonzero point $T\in E$ with
$T\oplus T=O$ such that $\tau(P)=P\oplus T$ for every $P\in E$.
\end{lemma}\begin{proof}
Write $\ominus$ for subtraction in the group law.
Set $T=\tau(O)$. The automorphism $\alpha(P)=\tau(P)\ominus T$
fixes $O$, so it is a group automorphism
\cite[Chapter II, Proposition 1.5]{MilneEllipticCurves}.
If $\alpha\ne\mathrm{id}$, then the map
$P\mapsto P\ominus\alpha(P)$ is a nonzero group endomorphism,
hence a nonconstant morphism from the projective curve $E$
to itself and therefore surjective. There would consequently
exist $P$ with $P\ominus\alpha(P)=T$, giving $\tau(P)=P$,
a contradiction. Thus $\tau(P)=P\oplus T$. The identity
$\tau^2=\mathrm{id}$ gives $T\oplus T=O$, and the absence
of fixed points gives $T\ne O$.
\end{proof}
\begin{proof}[Proof of \Cref{lem:genusone}]
Given the data in the hypothesis, we prove there exists a tuple $(A,B,r,s,h,e,f)$ and functions $u,\widetilde u, v, \widetilde v$ on $E$ satisfying  the fixed-length and coupler-point equations of a four-bar:
\begin{center}
$\begin{array}{cc}
u \widetilde u = r,\quad \quad \quad  v \widetilde v = s, \quad \quad \quad(B_z-A_z+v-u)(B_w-A_w+\widetilde v - \widetilde u) = h \\ z=A_z+u+e(B_z-A_z+v-u),\quad \quad \quad\quad w = A_w + \widetilde u +f (B_w-A_w+\widetilde v - \widetilde u)
\end{array}$
\end{center}
First, we define three functions $a,b,$ and $c$ on $E$ whose poles together will cover the poles $P_0,P_1,P_2$ of $z$ and whose zeros together cover the poles $\tau(P_0),\tau(P_1),\tau(P_2)$ of $w$.  Let $\oplus$ be addition on $E$ after choosing an origin $O \in E$. Then $\tau(P_0) \oplus \tau(P_1) = P_0 \oplus P_1$ since, by
\Cref{lem:tau_in_elliptic}, 
\begin{align*}
\tau(P_0)\oplus\tau(P_1)
&=P_0 \oplus T \oplus P_1 \oplus T \quad \quad \text{ for some } T \text{ with } T\oplus T = O\\ 
&=P_0\oplus P_1\oplus T\oplus T \\
&=P_0\oplus P_1. 
\end{align*}
On an elliptic curve, divisors of the same degree are linearly
equivalent precisely when their sums in the group law agree \cite[Chapter I, Proposition 4.10]{MilneEllipticCurves}.
Consequently, there are rational functions $a,b$, and $c$ with
\begin{align*}
\operatorname{div}(a)
&=\tau(P_0)+\tau(P_1)-P_0-P_1 \\
\operatorname{div}(b)
&=\tau(P_0)+\tau(P_2)-P_0-P_2 \\
\operatorname{div}(c)
&=\tau(P_1)+\tau(P_2)-P_1-P_2. 
\end{align*}
Each function is determined up to a nonzero scalar. 
Write
\[
D_z=P_0+P_1+P_2,\qquad
D_w=\tau(P_0)+\tau(P_1)+\tau(P_2).
\]By Riemann--Roch
\cite[Chapter I, Theorem 4.13]{MilneEllipticCurves},
an effective divisor $D$ of positive degree on a smooth
projective genus-one curve satisfies $\dim L(D)=\deg D$,
where $L(D)$ is the vector space of rational functions on $E$
whose poles are bounded by $D$. Thus, 
\[
\dim L(D_z)=\dim L(D_w)=3.
\]
The functions $a,b,c$ lie in $L(D_z)$ and are linearly independent:
at $\tau(P_0),\tau(P_1),\tau(P_2)$, respectively, only $c$, only $b$,
and only $a$ is nonzero. They therefore form a basis of $L(D_z)$.
Similarly, evaluation at $P_0,P_1,P_2$ shows that
$a^{-1},b^{-1},c^{-1}$ form a basis of $L(D_w)$.

Expressing the constant function in these bases gives
\[
1=\lambda_a a+\lambda_b b+\lambda_c c,\qquad
1=\frac{\mu_a}{a}+\frac{\mu_b}{b}+\frac{\mu_c}{c}.
\]Evaluating the first identity at $\tau(P_0),\tau(P_1),\tau(P_2)$
and the second at $P_0,P_1,P_2$ shows that all six coefficients
are nonzero. Absorbing the $\lambda$ coefficients into the functions $a,b,c$ gives
\[
1=a+b+c,\qquad
1=\frac{\kappa_a}{a}+\frac{\kappa_b}{b}+\frac{\kappa_c}{c},
\]
where $\kappa_a,\kappa_b,\kappa_c$ are nonzero constants.
The prescribed coordinate functions have unique expressions
\[
z=z_a a+z_b b+z_c c,\qquad
w=w_a\frac{\kappa_a}{a}
 +w_b\frac{\kappa_b}{b}
 +w_c\frac{\kappa_c}{c}.
\]
Substituting $b=1-a-c$ and $\frac{\kappa_b}{b} = 1- \frac{\kappa_a}{a}-\frac{\kappa_c}{c}$, respectively, these become
\[
z=z_b+(z_a-z_b)a+(z_c-z_b)c, \quad \quad 
w=w_b+(w_a-w_b)\frac{\kappa_a}{a}
     +(w_c-w_b)\frac{\kappa_c}{c}.
\]
These expressions suggest taking the fixed pivots to be $
A=(z_b,w_b)$ and $B=(z_a,w_a)$.
Write $L_z=z_a-z_b$ and $L_w=w_a-w_b$ for the coordinates
of the ground vector $\overrightarrow{AB}$.
Both are nonzero: otherwise $z$ would have no pole at $P_0$,
or $w$ would have no pole at $\tau(P_0)$.

Define the side-link coordinates by
\[
u=L_z a,\qquad \widetilde u=L_w\frac{\kappa_a}{a},
\qquad
v=-L_z b,\qquad \widetilde v=-L_w\frac{\kappa_b}{b}.
\]
Then the moving pivots
\[
D=A+(u,\widetilde u),\qquad C=B+(v,\widetilde v)
\]
satisfy
\[
\overrightarrow{DC}
=\left(L_z+v-u,\;L_w+\widetilde v-\widetilde u\right)
=\left(L_zc,\;L_w\frac{\kappa_c}{c}\right).
\]
Thus the three moving links have constant squared lengths
\[
r=L_zL_w\kappa_a,\qquad
s=L_zL_w\kappa_b,\qquad
h=L_zL_w\kappa_c.
\]
Finally, setting
\[
e=\frac{z_c-z_b}{L_z},\qquad
f=\frac{w_c-w_b}{L_w}
\]
gives
\[
z=A_z+u+e(L_z+v-u),\qquad
w=A_w+\widetilde u+f(L_w+\widetilde v-\widetilde u).
\]
These are the required coupler-point equations for the mechanism
$(A,B,r,s,h,e,f)$.
\end{proof}
\end{document}